\documentclass[12pt]{amsart}
\usepackage{amsmath,amssymb,euscript,mathrsfs, eufrak}
\usepackage[all]{xy}
\usepackage{enumerate}

\pushQED{\qed}

\newcommand{\define}{\textbf}

\renewcommand{\setminus}{\smallsetminus}
\renewcommand{\phi}{\varphi}

\renewcommand{\tilde}{\widetilde}
\renewcommand{\hat}{\widehat}
\renewcommand{\bar}{\overline}
\renewcommand{\L}{\mathbb{L}}
\newcommand{\D}{\mathbb{D}}
\newcommand{\HS}{\mathfrak{hs}}
\newcommand{\C}{\mathbb{C}}
\newcommand{\Q}{\mathbb{Q}}
\newcommand{\R}{\mathbb{R}}
\newcommand{\N}{\mathbb{N}}
\newcommand{\Z}{\mathbb{Z}}
\renewcommand{\P}{\mathbb{P}}
\newcommand{\A}{\mathbb{A}}
\renewcommand{\O}{\mathcal{O}}
\newcommand{\cp}{\mathcal{P}}

\newcommand{\K}{\mathbb{K}}

\newcommand{\cS}{\mathcal{S}}
\newcommand{\cG}{\mathcal{G}}
\newcommand{\cH}{\mathcal{H}}

\def\<{\ensuremath{\langle}}
\def\>{\ensuremath{\rangle}}

\newcommand{\cHp}{{}^p\mathcal{H}}
\newcommand{\lc}{l^*}

\newcommand{\Gr}{Gr}

\newcommand{\X}{\mathcal{X}}

\DeclareMathOperator{\inter}{int}
\DeclareMathOperator{\IH}{IH}

\DeclareMathOperator{\Aut}{Aut}

\DeclareMathOperator{\Spec}{Spec}
\DeclareMathOperator{\ord}{ord}

\DeclareMathOperator{\pt}{pt}
\DeclareMathOperator{\Hdg}{Hdg}
\DeclareMathOperator{\lk}{lk}

\DeclareMathOperator{\Int}{Int}

\DeclareMathOperator{\Trop}{Trop}

\DeclareMathOperator{\prim}{prim}
\DeclareMathOperator{\Lef}{Lef}

\DeclareMathOperator{\Var}{Var}

\DeclareMathOperator{\init}{in}

\DeclareMathOperator{\ver}{vert}
\DeclareMathOperator{\vol}{vol}
\DeclareMathOperator{\gen}{gen}

\DeclareMathOperator{\rec}{rec}

\DeclareMathOperator{\Conv}{Conv}
\DeclareMathOperator{\st}{st}

\newtheorem{theorem}{Theorem}[section]
\newtheorem{lemma}[theorem]{Lemma}
\newtheorem{proposition}[theorem]{Proposition}
\newtheorem{corollary}[theorem]{Corollary}

\theoremstyle{definition}
\newtheorem{definition}[theorem]{Definition}
\newtheorem{remark}[theorem]{Remark}
\newtheorem{example}[theorem]{Example}

\newcommand{\excise}[1]{}

\begin{document}

\title[]{%Computing cohomology 
%Limit mixed Hodge structures \\
%via
%tropical geometry
%The Tropical Motivic Nearby Fiber
%Tropical geometry and the Hodge theory of hypersurfaces
%Tropical Geometry and the Motivic Nearby Fiber II Should change again?
Tropical geometry, the motivic nearby fiber and  limit mixed Hodge numbers of hypersurfaces 
}
\author{Eric Katz}
\address{Department of Combinatorics \& Optimization, University of Waterloo, 200 University Avenue West, Waterloo, ON, Canada N2L 3G1} \email{eekatz@math.uwaterloo.ca}

\author{Alan Stapledon}
\address{Department of Mathematics\\University of Sydney\\ NSW, Australia 2006}
\email{alan.stapledon@sydney.edu.au}

\keywords{tropical geometry, monodromy, Hodge theory, polytopes, Ehrhart theory, intersection cohomology}
\date{}
\thanks{}

\begin{abstract}

The motivic nearby fiber is an invariant obtained from degenerating a complex variety over a disc. It specializes to the Euler characteristic of the original variety but also contains information on the variation of Hodge structure associated to the degeneration which is encoded as a limit mixed Hodge structure. However, this invariant is difficult to compute in practice. 
Using the techniques of tropical geometry we present a new formula for the motivic nearby fiber. Moreover, since there is a range of available software implementing the main algorithms in tropical geometry, our formula can be computed in practice. We specialize to the case of families of sch\"{o}n complex hypersurfaces of tori where we provide explicit formulas describing the action of the unipotent part of monodromy on the graded pieces (with respect to the Deligne weight filtration) of the cohomology with compact supports.  These families are described combinatorially by a polyhedral subdivision of the associated Newton polytope.  We develop new mixed Hodge theory-inspired combinatorial invariants of such subdivisions, among them the `refined limit mixed $h^*$-polynomial'.   These invariants are related to Stanley's combinatorial study of subdivisions:
in a companion combinatorial paper whose results are applied here, we situate our invariants in Stanley's theory where they become multi-variable extensions of his invariants. Our results generalize work of Danilov and Khovanski{\u\i} and Batyrev and Borisov on the Hodge numbers of hypersurfaces.   We also present analogous formulas describing the action of the unipotent part of monodromy on the intersection cohomology groups of a family of sch\"on hypersurfaces of a projective toric variety.
\end{abstract}

\maketitle
\tableofcontents

\section{Introduction}\label{s:intro}

Let $\O$ be the ring of germs of analytic functions in $\C$ in a neighborhood of the origin, and let $\K$ be its field of fractions.  A variety $X$ over $\K$ is naturally interpreted as a family of complex varieties $f:X\rightarrow\D^*$ where $\D^*$ is a small punctured disc about the origin over which $X$ is defined.  After possibly shrinking $\D^*$, we may assume that $X \rightarrow \D^*$ is a locally trivial fibration, and  we fix a non-zero fiber $X_{\gen} := f^{-1}(t)$ for some $t \in \D^*$.  

Our goal is to compute an important invariant of $X$ called the \define{motivic nearby fiber} $\psi_X = \psi_f$, that was introduced by Denef and Loeser \cite{DLGeometry} and contains information about the extension of $f$ to a family over the whole complex disc $\D$. Moreover, the motivic nearby fiber specializes to the limit Hodge-Deligne polynomial of $X$ and to both the $\chi_y$-characteristic and Euler characteristic of $X_{\gen}$.

The motivic nearby fiber is `additive' in the following sense. For any field $k$,
the \define{Grothendieck ring} $K_0(\Var_k)$ of algebraic varieties over $k$ is the free abelian group generated by isomorphism classes $[V]$ of varieties $V$ over $k$, modulo  the relation 
\begin{equation*}
[V] = [U] + [V \setminus U], 
\end{equation*}
whenever $U$ is an open subvariety of $V$.  Multiplication is defined by 
\[
[V] \cdot [W] = [V \times W].
\]
 We will follow the convention that $\L := [\A^1]$. 
A \define{motivic invariant} over $k$ is a ring homomorphism $K_0(\Var_k) \rightarrow R$, for some ring $R$. 
The motivic nearby fiber is a ring homomorphism 
\begin{equation}\label{e:mmap}
\psi: K_0(\Var_\K) \rightarrow K_0(\Var_\C), \: \psi([X]) = \psi_X. 
\end{equation}

We briefly recall the construction of the motivic nearby fiber, and refer the reader to \cite{BitMotivic} for details. 
A result of Bittner \cite{BitUniv} implies that if $k$ has characteristic zero, then $K_0(\Var_k)$ is generated by the classes of smooth, proper varieties. If $X$ is smooth and proper, then by \cite{KKMS}
 there exists a \define{semi-stable reduction} of $X$. That is, 
after possibly
 pulling back the family $f: X \rightarrow \D^*$ by a map $\D^*\rightarrow \D^*$ ramified over the puncture, there exists an extension of $f$ defined over $\D$ such that the central fiber is a reduced, simple 
 normal crossings divisor with irreducible components $\{ D_i \}_{i \in \{ 1, \ldots, r \}}$. If for every non-empty subset  $I \subseteq \{ 1, \ldots ,  r \}$, we set $D_I^\circ = \cap_{i \in I} D_i \setminus \cup_{j \notin I} D_j$,  then 
\begin{equation*}
\psi_X = \sum_{\emptyset \ne I \subseteq \{ 1, \ldots ,  r \}} [D_I^\circ ] (1 - \mathbb{L})^{|I| - 1}.
\end{equation*}
We will give an approach to computing 
the motivic nearby fiber via tropical geometry.

A result of Luxton and Qu  \cite[Theorem~6.11]{LuxtonQu} that was conjectured by Tevelev in \cite{TevComp},  states that every variety $X$ over $\K$ contains an %dense, 
open, very affine subvariety $X^\circ$ that is \define{sch\"on} in the sense of Tevelev \cite[Definition~1.1]{TevComp}. Here $X^\circ$ being very affine means that  it can be embedded 
as a closed subvariety of $(\K^*)^n$ for some $n$, defined by an ideal $I \subseteq \K[x_{1}^{\pm 1},\ldots, x_n^{\pm 1}]$. In this case, $X^\circ$ being sch\"on means that for 
for every $w \in \R^n$, the corresponding initial degeneration $\init_w X^\circ$ defined by the ideal $\init_w I := ( \init_w(f) \mid f \in I) \subseteq \C[x_{1}^{\pm 1},\ldots, x_n^{\pm 1}]$ of initial degenerations is a 
smooth subvariety of $(\C^*)^n$ \cite[Prop 3.9]{HelmKatz}. For a more geometric description of $\init_w X^\circ$, we refer the reader to Section~\ref{s:proof}. 
The notion of %a  
sch\"onness of a hypersurface of a complex torus  was introduced by Khovanski{\u\i} in \cite{KhoNewton} as a hypersurface \emph{non-degenerate} with respect to its Newton polytope.

Luxton and Qu's result immediately  implies that the Grothendieck ring $K_0(\Var_\K)$ is generated by sch\"on subvarieties of tori. In particular, to describe the motivic nearby fiber, in principle, we may reduce to the case of a sch\"on subvariety of a torus.
\emph{In what follows, we will always assume that $X^\circ \subseteq (\K^*)^n$ is sch\"on.}
\begin{definition} The tropical variety $\Trop(X^\circ)$ associated to $X^\circ$ is the set of points %closure in $\R^n$ of the set of points
\[\{w\in\R^n|\init_w X^\circ\neq \emptyset\}.\]
%\[\{w\in\Q^n|\init_w X^\circ\neq \emptyset\}.\]
\end{definition}

The tropical variety $\Trop(X^\circ)$ can be given a rational polyhedral structure $\Sigma$ such that initial degeneration at $w \in \Trop(X^\circ)$ only depends on the cell containing $w$ in its relative interior (this follows from 
\cite[Theorem~1.5]{LuxtonQu}). Hence
for every cell $\gamma$ of $\Sigma$, we may define $[\init_{\gamma} X^\circ] := [\init_w X^\circ] \in K_0(\Var_\C)$ for any 
%$w \in \Q^n$
$w \in \R^n$ in the relative interior %$\Int{\gamma}$ 
of $\gamma$. Our main result is as follows:

\begin{theorem} \label{t:comp} Let $X^\circ \subseteq (\K^*)^n$ be a sch\"{o}n closed subvariety and let $\Sigma$ be a rational polyhedral structure on $\Trop(X^\circ)$.  
Then the motivic nearby fiber $\psi_{X^\circ}$ is given by
\[
\psi_{X^\circ} =   \sum_{\substack{\gamma\in\Sigma\\\gamma \operatorname{bounded}}}\, (-1)^{\dim \gamma} [\init_\gamma X^\circ]. 
\] 
\end{theorem}

A key point is that there exist explicit algorithms to compute both the initial degenerations of $X^\circ$ and its tropical variety with a choice of rational polyhedral structure. Moreover, there is a range of available software  that implements these algorithms \cite{GFan,tropicallib}.   Hence given any variety over $\K$, if one is able to produce a stratification into locally closed, very affine sch\"on subvarieties, as guaranteed by Luxton and Qu's result, then the above theorem gives a practical approach to computing the motivic nearby fiber. 

\begin{example}\label{e:concrete}
For a concrete example, let $t$ be a local co-ordinate on $\D$, and let 
\[
C^\circ = \{ (x,y) \in (\K^*)^2 \mid t(1 + x^4 + y^4) + xy(1 + x + y) = 0 \}. 
\]
Then $C^\circ_{\gen}$ is a genus  $3$ curve with $12$ points removed. %In particular, $e(C^\circ_{\gen}) = -16$. 
The corresponding tropical variety has a polyhedral structure with four vertices $v_1 = (1,0)$, $v_2 = (0,1)$, $v_3 = (-1,-1)$ and $v = (0,0)$, 
six bounded edges joining each pair of vertices, and 3 unbounded edges emanating from each $v_i$ in the direction of $v_i$. 
The initial degeneration at each $v_i$, at $v$, and at each bounded edge is isomorphic to  $\A^1$ minus $6$, $2$ and $1$ point respectively.
Theorem~\ref{t:comp} then implies that
\[
\psi_{C^\circ} = 3(\L - 6) + (\L - 2) - 6(\L - 1)  = -14 - 2 \L . 
\]
\end{example}

We provide a proof of Theorem~\ref{t:comp} in Section~\ref{s:proof}. The theorem immediately gives expressions for the motivic nearby fiber of  various partial compactifications of $X^\circ$ that are not smooth in general (Corollary~\ref{c:partial}). 
In particular, it generalizes Theorem~5.1 in \cite{KatzStapledon} in the case of smooth, compactifications (see Remark~\ref{r:generalize}). 

Observe that by composing the motivic nearby fiber map \eqref{e:mmap} with a motivic invariant over $\C$, we obtain a new motivic invariant over $\K$, to which we may apply our formula. 
In particular, as described in detail in Section~\ref{s:motivic}, 
if 
$E: K_0(\Var_\C) \rightarrow \Z[u,v] $ denotes the Hodge-Deligne map, then  we obtain
a series of well-known invariants:
\begin{equation}\label{e:bottom}
K_0(\Var_\K) \xrightarrow{ \psi } K_0(\Var_\C) \xrightarrow{E} \Z[u,v] \xrightarrow{v \mapsto 1} \Z[u] \xrightarrow{u \mapsto 1}  \Z. 
\end{equation}
 For any variety $X$ over $\K$, the polynomial $E(X_\infty;u,v) := E(\psi([X]))$ is called the \define{limit Hodge-Deligne polynomial} of $X$, and encodes information on the variation of mixed Hodge structures
 of the family $X \rightarrow \D^*$. 
The specialization obtained by setting $v = 1$ 
is the \define{$\chi_y$-characteristic} $E(X_{\gen};u,1) = E(X_\infty;u,1)$ of $X_{\gen}$, and encodes information about the Hodge filtration on the cohomology with compact supports of $X_{\gen}$. 
Finally, the specialization $e(X_{\gen}) =  E(X_{\gen};1,1)$ is the familiar \define{topological Euler characteristic} of $X_{\gen}$. 
Theorem~\ref{t:comp} immediately provides formulas for these invariants in the case when $X^\circ$ is sch\"on. % we obtain the following formula for the Euler characteristic of $X^\circ_{\gen}$. 

\begin{corollary}\label{c:main}
Let $X^\circ \subseteq (\K^*)^n$ be a sch\"{o}n closed subvariety and let $\Sigma$ be a rational polyhedral structure on $\Trop(X^\circ)$.  
%Let $X^\circ \subseteq (\K^*)^n$ be sch\"{o}n, let $\Sigma$ be a rational polyhedral structure on $\Trop(X^\circ)$, and 
Let $\ver(\Sigma)$ denote the set of vertices of $\Sigma$. 
If we fix a non-zero fiber $X^\circ_{\gen} := f^{-1}(t)$ for some $t \in \D^*$, then the limit Hodge-Deligne polynomial of $X^\circ$ is given by
\[
E(X^\circ_\infty;u,v) =   \sum_{\substack{\gamma\in\Sigma\\\gamma \operatorname{bounded}}}\, (-1)^{\dim \gamma} E(\init_\gamma X^\circ;u,v), 
\]
the $\chi_y$-characteristic of $X^\circ_{\gen}$ is given by
\[
E(X^\circ_{\gen};u,1) = \sum_{\substack{\gamma\in\Sigma\\\gamma \operatorname{bounded}}}\, (-1)^{\dim \gamma} E(\init_\gamma X^\circ;u,1), 
\]
and the 
Euler characteristic  $e(X^\circ_{\gen})$ of $X^\circ_{\gen}$ is given by
\[
e(X^\circ_{\gen}) =   \sum_{\gamma\in \ver(\Sigma)}\, e(\init_\gamma X^\circ). 
\] 
\end{corollary}

Here the last equality follows from the fact that  the Euler characteristic $e(\init_\gamma X^\circ)$ is zero unless $\gamma$ is a vertex of $\Sigma$ (see \eqref{e:toruscalc}). 

As discussed above, this corollary provides a strategy to compute any of these invariants. For example, if one wants to compute the Euler characteristic of a complex variety $V$, then 
if one can find a stratification of $V$ into locally closed pieces, each of which can be realized as the general fiber of a sch\"on degeneration, then the above corollary reduces the
problem to finding the Euler characteristic of a set of `simpler' complex varieties.  

Before presenting our main application, we introduce a new motivic invariant over $\K$ (see Section~\ref{s:motivic} for details). 
Given a variety $X$ over $\K$, consider the complex cohomology with compact supports $H^m_c  (X_{\gen})$ of the fiber $X_{\gen}$. Then $H^m_c  (X_{\gen})$ admits three natural filtrations. Firstly, since it 
is a complex variety, it admits a decreasing filtration $F^\bullet$ called the \define{Hodge filtration} and an increasing filtration $W_\bullet$ called the 
\define{Deligne weight filtration}. Secondly, the monodromy map $T: H^m_c  (X_{\gen}) \rightarrow H^m_c  (X_{\gen})$ can be written as $T = T_sT_u$, where $T_s$ is semi-simple and 
$T_u$ is unipotent, and we may consider the action of the nilpotent operator $N = \log T_u$ on $H^m_c  (X_{\gen})$.  A result of Steenbrink and Zucker \cite{SZ} and El Zein \cite{ElZein} states that 
$H^m_c  (X_{\gen})$ admits an increasing filtration $M_\bullet$ called the 
\define{monodromy weight filtration}, such that the filtration induced by $M_\bullet$ on  the quotient $\Gr_r^W H_c^m (X_{\gen})$ encodes the Jordan block structure of the induced action of $N$ on 
$\Gr_r^W H_c^m (X_{\gen})$. Here, we use $H_c^m(X_{\infty})$ to mean compactly supported cohomology equipped with the Hodge, monodromy, (and possibly also weight) filtrations.
We will refer to the corresponding invariants %We will be interested in computing the corresponding invariants 
\[h^{p,q,r}(H_c^m(X_{\infty})) =\dim(\Gr_F^p\Gr^{M}_{p+q}\Gr_r^W H_c^m(X_\infty)),\]
%which we refer to 
as the \define{refined limit mixed Hodge numbers}. Summing over $q$ or $r$ recovers the (usual) mixed  Hodge numbers and the limit mixed  Hodge numbers of $H_c^m (X_{\gen})$ respectively (see  \eqref{e:limitnumber} and \eqref{e:number}). 
%encode the action of the unipotent part of monodromy on the pieces of the (usual) mixed Hodge structure on the complex variety $X_{\gen}$. 
We define a polynomial called the  \define{refined limit Hodge-Deligne polynomial} by 
\[
E(X_\infty; u,v,w) = \sum_{p,q,r} \sum_{m} (-1)^m h^{p,q,r}(H_c^m(X_{\infty})) u^p v^q w^r, 
\]
and show that we have an induced motivic invariant 
\[
\bar{E}: K_0(\Var_\K) \rightarrow \Z[u,v,w], \: [X] \mapsto  E(X_\infty; u,v,w). 
\]
We now have a commutative diagram of motivic invariants 
\[\xymatrix{
K_0(\Var_\K) \ar[r]^{\bar{E}} \ar[d]^\psi  &\Z[u,v,w]  \ar[r]^{\substack{u \mapsto uw^{-1} \\ v \mapsto 1}}   \ar[d]^{w \mapsto 1}  & \Z[u,w] \ar[d]^{w \mapsto 1}  &\\
K_0(\Var_\C) \ar[r]^{E} &\Z[u,v] \ar[r]^{v \mapsto 1} & \Z[u]  \ar[r]^{u \mapsto 1}  & \Z,
}\] 
where the first vertical arrow together with the lower horizontal row coincide with \eqref{e:bottom},  and we have corresponding invariants
\[\xymatrix{
[X] \ar[r] \ar[d] &  E(X_\infty; u,v,w)  \ar[r] \ar[d] & E(X_{\gen}; u,w) \ar[d] &\\
\psi_X \ar[r] & E(X_\infty; u,v) \ar[r]& E(X_{\gen}; u,1)   \ar[r]   & e(X_{\gen}),
}\] 
where $E(X_{\gen}; u,w)$ is the \define{Hodge-Deligne polynomial} of $X_{\gen}$. One may think that every successive specialization forgets about a filtration in the following sense: the invariants 
$E(X_\infty; u,v,w)$, $E(X_\infty; u,v)$, $E(X_{\gen}; u,w)$ and $E(X_{\gen}; u,1)$ encode information about the filtrations $(F^\bullet, W_\bullet, M_\bullet)$, $(F^\bullet, M_\bullet)$, $(F^\bullet, W_\bullet)$ and
$F^\bullet$ respectively. 

For the remainder of the introduction, we assume  that $X^\circ \subseteq (\K^*)^n$ is a sch\"{o}n hypersurface. In this case, the Hodge-Deligne polynomial  $E(X^\circ_{\gen}; u,w)$ encodes precisely the (usual) mixed Hodge numbers of $X^\circ_{\gen}$, and its computation is a classical problem. Indeed, an algorithm to compute the mixed Hodge numbers of a sch\"on hypersurface of a complex torus was given by Danilov and Khovanski{\u\i} in \cite{DKAlgorithm}. Much later, using deep results from intersection cohomology, a combinatorial formula was given by Batyrev and Borisov, and was the key technical result in their construction of mirror Calabi-Yau varieties  in \cite{BBMirror}. A cleaner combinatorial formula was later given by 
Borisov and Mavlyutov in \cite{BMString}. Finally, a combinatorial proof of the Borisov-Mavlyutov formula was given by the second author in  \cite{StaMirror}, as part of work giving a representation-theoretic generalization. 

Our main application is a combinatorial formula for the refined limit mixed Hodge numbers of the sch\"on hypersurface $X^\circ$. 
In this case, this is equivalent to giving a combinatorial formula for the refined limit Hodge-Deligne polynomial $E(X^\circ_\infty; u,v,w)$. 
In particular, by specializing, we obtain  a combinatorial formula for the limit mixed Hodge numbers of $X^\circ$. Our result also specializes to 
give  the Borisov-Mavlyutov formula for the usual mixed Hodge numbers of $X^\circ_{\gen}$. Although we make use of the strategy of the Danilov-Khovanski{\u\i} algorithm, 
our proof is self-contained and only relies on Theorem~\ref{t:comp} together with
some new combinatorics.  
In particular, in Section~\ref{s:chiy}, using the theory of valuations of polytopes (see, for example, \cite{McMullen}), we present a new proof of a formula 
of Danilov-Khovanski{\u\i} \cite[Section~4]{DKAlgorithm} for the $\chi_y$-characteristic of $X^\circ_{\gen}$. 
Since the necessary combinatorial results are  involved, and we expect them to be of  outside interest, we will only quote them as needed and defer all proofs and discussion to \cite{otherpaper}.  We will mention that some of these results build on the work of Stanley \cite{StaSubdivisions}, together with recent work of  Athanasiadis and  Savvidou \cite{AthFlag,AthSymmetric}, and Nill and Schepers \cite{NSCombinatorial}.

As explained in Section~\ref{ss:tropical}, we may associate to $X^\circ$ its corresponding Newton polytope $P$ together with a corresponding regular, lattice polyhedral subdivision $\cS$. 
In \cite[Section~9]{otherpaper}, we introduce a combinatorial invariant $h^*(P,\cS; u,v,w) \in \Z[u,v,w]$  called the \define{refined limit mixed $h^*$-polynomial} of $(P,\cS)$, 
 that only depends on the poset structure of $\cS$, together with the number of lattice points in all 
dilates of all cells of $\cS$. This invariant has several interesting specializations. In particular, $h^*(P,\cS; u,1,1) = h^*(P; u) $ is the usual $h^*$-polynomial of $P$, encoding the number of lattice points in all dilates of $P$ \cite{BRComputing}.  

\begin{theorem}\label{t:mainhyper}
Let $X^\circ \subseteq (\K^*)^n$ be a sch\"{o}n hypersurface, with associated Newton polytope and polyhedral subdivision $(P,\cS)$ and $\dim P = n$. Then the refined limit Hodge-Deligne polynomial of $X^\circ$ is given by
\[
uvw^2E(X_\infty^\circ;u,v,w) = (uvw^2 - 1)^{\dim P} + (-1)^{\dim P + 1}h^*(P,\cS;u,v,w).
\]
\end{theorem}

As discussed above, Theorem~\ref{t:mainhyper} immediately gives explicit combinatorial formulas for the refined limit mixed Hodge numbers and limit mixed Hodge numbers of $X^\circ$ (see Corollary~\ref{c:explicit}). In particular, we deduce that these invariants 
only depend on the pair $(P,\cS)$, and not on the specific choice of $X^\circ$.  Our results allow one to compute the refined limit Hodge-Deligne polynomial of various compactifications of $X^\circ$ but not necessarily the refined limit mixed Hodge numbers as we elaborate in Remark~\ref{r:singular}. In Example~\ref{e:stringy}, we apply our results to obtain formulas for \define{stringy invariants} associated to families of Calabi-Yau varieties. 

\begin{example}\label{e:intro2}
When $n = 2$, $X^\circ \subseteq (\K^*)^2$ may be viewed as a family of non-compact, smooth curves. Let $(P,\cS)$ denote the corresponding pair consisting of a lattice polytope in a lattice $M$ together with a lattice polyhedral subdivision. In this case, Theorem~\ref{t:mainhyper} has the following explicit description. Let $\partial P$ and $\Int(P)$ denote the boundary and interior of $P$ respectively. Then the coefficients of 
$uvw^3$ and $uv^2w^3$ in $h^*(P,\cS;u,v,w)$ are respectively given by
\[
h^*_{0,0,1}(P,\cS) =  \sum_{\substack{F \in \cS, F \nsubseteq \partial P \\\ \dim F \le 1 }} \# (\Int(F) \cap M),
\]
\[
h^*_{0,1,1}(P,\cS) =   \sum_{\substack{F \in \cS \\\ \dim F = 2 }} \# (\Int(F) \cap M), 
\]
and one can compute
\[
E(X^\circ_\infty;u,v,w) = 1 - \# (\partial P \cap M) - h^*_{0,0,1}(P,\cS)(1 + uv)w - h^*_{0,1,1}(P,\cS)(u + v)w + uvw^2.
\]
If $X$ denotes the closure of $X^\circ$ in the toric variety over $\K$ corresponding to the normal fan of $P$, then $X$ may be viewed as a family of smooth, compact curves with
\[
E(X_\infty;u,v,w) = 1 - h^*_{0,0,1}(P,\cS)(1 + uv)w - h^*_{0,1,1}(P,\cS)(u + v)w + uvw^2.
\]
When $n = 3$ and $X^\circ$ may be viewed as a family of non-compact, smooth surfaces, an explicit description of $E(X^\circ_\infty;u,v,w)$ is given by Theorem~\ref{t:mainhyper} and Example~\ref{e:smallterms}. 
\end{example}

\begin{example}
Continuing with the explicit family of curves in Example~\ref{e:concrete}, the corresponding Newton polytope $P$ is the 
convex hull of $a_0 = (0,0)$, $a_1 = (4,0)$ and $a_2 = (0,4)$. Setting $b_0 = (1,1)$, $b_1 = (2,1)$ and $b_2 = (1,2)$, the lattice 
polyhedral subdivision $\cS$ has four maximal cells: $\{ a_i, a_j, b_i, b_j \}$ for $i \ne j$ and $\{ b_0, b_1, b_2 \}$.  
By Example~\ref{e:intro2}, 
\[
E(C^\circ_\infty;u,v,w) = -11 - 3(1 + uv)w + uvw^2.
\]
\end{example}

In the case when we have a family of varieties over a punctured curve, we also give an alternative approach to Theorem \ref{t:mainhyper} via intersection cohomology making use of the pure Hodge structure on the intersection cohomology of projective varieties.  By the use of the decomposition theorem of Beilinson, Bernstein, Deligne and Gabber \cite{BBD}, one can show that for certain stratifications, intersection cohomology admits a motivic formula if one includes terms accounting for the singularities in the normal cones to strata.  This idea is used in the computation of intersection cohomology of toric varieties (see, e.g. \cite{F}), in the work of 
 Batyrev and Borisov \cite{BBMirror}, and is developed in greater generality by Cappell, Maxim, and Shaneson \cite{IntCoh}.  Here, we observe that a motivic formula holds for the refined limit Hodge-Deligne polynomials for intersection cohomology with compact support (Theorem~\ref{t:sumoverstrata}), and deduce that the following corollary is equivalent to Theorem~\ref{t:mainhyper} (see Lemma~\ref{l:equiv}). 
 The degree of $h^*(P,\cS;u,v,w)$ as a polynomial in $w$ is at most $\dim P + 1$, and we denote the coefficient of 
$w^{\dim P + 1}$ by  $\lc(P, \cS; u,v)$ and call it the \define{local limit mixed $h^*$-polynomial}. 
 
 \begin{corollary}\label{c:mainintersection}
 Let $\K = \C(t)$ and 
let $X^\circ \subseteq (\K^*)^n$ be a sch\"{o}n hypersurface, with associated Newton polytope and polyhedral subdivision $(P,\cS)$ and $\dim P = n$. Let $X$ denote the closure of $X^\circ$ in the projective toric variety over $\K$ corresponding to the normal fan of $P$. Then the refined limit Hodge-Deligne polynomial associated to the intersection cohomology of $X$ is given by
\[
uvw^2E_{\inter}(X_\infty;u,v,w) = uvw^2E_{\inter,\Lef}(P;uvw^2) + (-1)^{\dim P + 1}\lc(P,\cS;u,v)w^{\dim P + 1},
\]
where 
\[
(t - 1)E_{\inter,\Lef}(P;t) = t^{\dim P}g([\emptyset,P]^*;t^{-1})- g([\emptyset,P]^*;t)
\]
is defined in terms of Stanley's g-polynomial (see Definition~\ref{d:g}). 
\end{corollary}

From the above corollary, one may deduce an explicit formula for the corresponding refined limit mixed Hodge numbers for intersection cohomology (see Corollary~\ref{c:intexplicit}). 
When $\K = \C(t)$,  we also present an alternative proof of Corollary~\ref{c:mainintersection} and hence of Theorem~\ref{t:mainhyper} using intersection cohomology. 
This proof extends the ideas of  Batyrev and Borisov's original proof of a formula for the usual mixed Hodge numbers of $X^\circ_{\gen}$  in \cite{BBMirror}.

\begin{example}\label{e:intro3}
As in Example~\ref{e:intro2}, let $X^\circ \subseteq (\K^*)^n$ be a sch\"{o}n hypersurface. Let $(P,\cS)$ denote the corresponding pair consisting of a lattice polytope together with a lattice polyhedral subdivision. Let $X$ denote the closure of $X^\circ$ in the toric variety over $\K$ corresponding to the normal fan of $P$. 
When $n = 2$, $X$ may be viewed as a family of compact, smooth curves, and $E_{\inter}(X_\infty;u,v,w) = E(X_\infty;u,v,w)$ is computed in Example~\ref{e:intro2}. When $n = 3$, 
$X$ may be viewed as a family of compact, possibly singular surfaces, and $E_{\inter}(X_\infty;u,v,w)$ is given explicitly by Corollary~\ref{c:mainintersection}, Example~\ref{e:smallterms} and 
the computation $E_{\inter,\Lef}(P;t) = 1+ \mu t + t^2$, where $\mu + 3$ is the number of facets of $P$.
\end{example}

%We provide a detailed description of the geometric invariants above in Section~\ref{s:motivic}. 
 \subsection{Organization of the paper} 

     This paper is structured as follows. In Section \ref{s:proof}, we review necessary background from tropical geometry, % and from our previous paper, 
     introduce our invariant $\psi_{(X^\circ,\Sigma,\Delta)}$ of partial compactifications of subvarieties of algebraic tori, and prove Theorem~\ref{t:comp}.  In Section \ref{s:motivic}, we discuss motivic invariants and the refined limit Hodge-Deligne polynomial.  Section \ref{s:combinatorial} introduces combinatorial invariants whose properties are established in \cite{otherpaper} and which are related to the refined limit Hodge-Deligne polynomial of hypersurfaces of algebraic tori in Section \ref{s:hypersurfaces}.  In Section~\ref{s:intcoh}, we derive a formula for the limit Hodge-Deligne polynomial of the intersection cohomology of a sch\"{o}n subvariety and use it to give an alternative proof of Theorem \ref{t:mainhyper}.

\medskip
\noindent
{\it Notation and conventions.}  
If $\P$ is a toric variety, then we let  $\P_\C$, $\P_\K$ and $\P_\O$ denote the corresponding toric variety over $\C$ and $\K$, and corresponding toric scheme over $\O$ respectively. 

\medskip
\noindent
{\it Acknowledgements.} 
We would like to thank Mark Andrea de Cataldo, Laurentiu Maxim and Gregory Pearlstein for valuable discussions.  Particular thanks should go to Patrick Brosnan who suggested the relevant mixed Hodge theory framework and to Benjamin Nill who introduced the authors to Stanley's subdivision theory and advocated its importance in establishing the relevant combinatorial theory.

\section{A tropical approach to the motivic nearby fiber %Proof of  Theorem~\ref{t:comp}
} \label{s:proof}

In this section, we present the proof of Theorem~\ref{t:comp}. %We will see that the result follows by induction on dimension from \cite[Theorem~5.1]{KatzStapledon}. 
We will continue with the notation of the introduction. In particular, $X^\circ \subseteq (\K^*)^n$ is a sch\"on subvariety, and $\Sigma$ is a rational polyhedral structure on $\Trop(X^\circ)$ that extends to a polyhedral subdivision of $\R^n$.   Such a $\Sigma$ exists by \cite[Prop 6.8]{LuxtonQu}.

We first recall the following toric interpretation of the initial degenerations of $X^\circ$, and refer the reader to \cite[Section~1]{HelmKatz} for details. 
One can define a toric scheme $\P(\Sigma)_\O$ over $\O$ from $\Sigma$. 
For a cell $\gamma$ of $\Sigma$, let $\rec(\gamma)$ denote the recession cone of $\gamma$. 
That is, $\rec(\gamma)$ is the unique cone such that there exists a bounded polytope $Q$ satisfying $\gamma = Q + \rec(\gamma)$.  
  By \cite{BGS}, the set of recession cones of $\Sigma$ forms the recession fan $\Delta$.
Note that the bounded cells of $\Sigma$ are precisely the cells whose recession cone is $\{0\}$.   
The generic fiber of $\P(\Sigma)_\O$ is the toric variety $\P(\Delta)_\K$.  
For cones $\tau$ in $\Delta$, let $U_\tau$ be the corresponding torus orbit of $\P(\Delta)_\K$.  
Cells $\gamma\in\Sigma$ correspond to torus orbits $U_\gamma$ contained in the central fiber of $\P(\Sigma)_\O$. 
We define $T_\gamma$ to be the torus fixing $U_\gamma$ pointwise. 
%There is an inclusion-reversing correspondence $F\mapsto \overline{U}_F$ between faces and orbit closures.  %In fact, each orbit closure $\overline{U}_F$ is the toric variety associated to the fan $\Delta_F=\Star_\Sigma(F)/(N_F)_\R$.  
%The torus orbits $U_F$ give a decomposition of the central fiber of $\P$.  

Let $\X$ denote the closure of $X^\circ$ in $\P(\Sigma)_\O$, and let $X_{\Delta}$ and $X_0$ denote the generic fiber and central fiber of $\X$ respectively.  
For cones $\tau$ in $\Delta$, let $X_\tau^\circ = X_{\Delta} \cap U_\tau$, so that $X_{\Delta}$ admits a stratification $X_{\Delta} = \cup_{\tau \in \Delta} X^\circ_{\tau}$.  
Similarly, for cells $\gamma$ in $\Sigma$, if we let $X_\gamma^\circ=\X\cap U_\gamma$, then $X_0 = \cup_{\gamma \in \Sigma} X^\circ_{\gamma}$.
%The central fiber $X_0$ of $\X$ can be written as a union of locally closed subschemes labelled by faces of $\Sigma$.
%Let $X_F^\circ=\X\cap U_F$.  %Let $X_F$ denote the closure of $X_F^\circ$ in $\X$.  

For each cone $\tau$ in $\Delta$, let $\R_\tau$ denote the linear span of $\tau$ and consider the projection $\pi_\tau: \R^n \rightarrow \R^n/\R_\tau$. Then $X_\tau^\circ$ is a sch\"on subvariety of  $U_\tau$, and its corresponding tropical variety has a polyhedral structure $\Sigma_\tau = \{ \pi_\tau(\gamma) \mid \tau \subseteq \rec(\gamma)  \}$.
% whose faces are obtained via 
%projection of the faces of $\Sigma$ whose corresponding recession cones are contained in $\tau$. 
In particular, the bounded cells of $\Sigma_\tau$ correspond to the cells of $\Sigma$ with recession cone $\tau$, and the recession fan $\Delta_\tau$ of $\Sigma_\tau$ is the \emph{star-quotient} of $\Delta$ by $\tau$ (see  \cite[Section~3.1]{FulIntroduction}). 

For $w$ in the relative interior of $\gamma$, the initial degeneration $\init_w X^\circ$ depends only on $\gamma$ because the closure of $X$ in $\P(\Sigma)_\O$ is a tropical compactification by \cite[Theorem~1.5]{LuxtonQu}.  Moreover, $\init_w X^\circ$ is invariant under the torus $T_\gamma$.   Moreover, there is a non-canonical isomorphism
\[\init_w X^\circ\cong T_\gamma\times X_\gamma^\circ. %(\X\cap U_F).
\]
Then we have
\begin{equation}\label{e:toruscalc}
[\init_w X^\circ] = [T_\gamma][X_\gamma^\circ] = [(\C^*)^{\dim \gamma}] [X_\gamma^\circ] = [\C^*]^{\dim \gamma}[X_\gamma^\circ]  = (\L - 1)^{\dim \gamma}[X_\gamma^\circ].  
\end{equation}

%If $F$ is a top-dimensional cell of $\Sigma$, then $m(F)$, the {\em multiplicity} of $F$ is the length of the zero-dimensional scheme $\init_w X^\circ/T_F$.  With these multiplicities $\Sigma$ is a balanced weighted rational polyhedral complex. 
%Let $\Delta$ denote the recession fan of $\Sigma$.  Then $X_{\K}\subset X(\Sigma)$ is stratified by its intersection with torus orbits.  For $\tau\in\Delta$, let $X_\tau$ be the intersection of $X$ with the orbit closure $V(\tau)$.  Let $X_\tau^\circ$ be the intersection of $X_\tau$ with the dense open torus of $V(\tau)$.  %Note that $V(\tau)$ is the toric variety associated to the star-quotient $\Delta_\tau$ as noted in \cite[Section~3.1]{FulIntroduction}.

%If we set
%\[
%\psi_{(X^\circ,\Sigma,\{0\})} :=   \sum_{\substack{F\in\Sigma\\F \operatorname{bounded}}}\, [X^\circ_F](1 - \L)^{\dim F}. 
%\]

For any subfan $\Delta'$ of $\Delta$, we define 
\[
\psi_{(X^\circ,\Sigma,\Delta')} :=   \sum_{\substack{\gamma\in\Sigma\\\ \rec(\gamma) \in \Delta'}}\, [X^\circ_\gamma](1 - \L)^{\dim \gamma-\dim(\rec(\gamma))}.
\]
It follows from \eqref{e:toruscalc} that 
%We conclude that  
Theorem~\ref{t:comp} is equivalent to the following:
\begin{equation}\label{e:open}
\psi_{X^\circ} =  \psi_{(X^\circ,\Sigma,\{0\})}. 
\end{equation}
%For each cone $\tau$ in $\Delta$, let $\{\tau\}$ denote the corresponding subfan.
It follows from the above description that for each cone $\tau$ in $\Delta$,
\[
\psi_{(X_\tau^\circ,\Sigma_\tau,\{0\})} = \sum_{\substack{\gamma\in\Sigma\\\ \rec(\gamma) = \tau}}\, [X^\circ_\gamma](1 - \L)^{\dim \gamma-\dim \tau}.
\]
Hence, we have the relation
\begin{equation}\label{e:summotivic}
\psi_{(X^\circ,\Sigma,\Delta')} = \sum_{\tau \in \Delta'} \psi_{(X_\tau^\circ,\Sigma_\tau,\{0\})}.
\end{equation}

A priori, $\psi_{(X^\circ,\Sigma,\{0\})}$ depends on $\Sigma$.  Our first step is to show that is independent of $\Sigma$ below. We first recall the following lemma that was proved in \cite[Lemma~3.4]{KatzStapledon}. We provide a more concise proof below.

\begin{lemma} \label{l:euler}
Let $P$ be a $n$-dimensional polytope and let $Q$ be a proper (possibly empty) face of $P$. Let $\cS$ be a polyhedral subdivision of $P$.  Then
\[\sum_{ \substack{  F \in \cS, F\cap Q=\emptyset\\\Int(F) \subseteq \Int(P)}} (-1)^{\dim F}=\left\{\begin{array}{cl} (-1)^d & \text{if } Q =  \emptyset  \\ 0 & \text{otherwise}. \end{array}\right.\]
\end{lemma}

\begin{proof}

Let $P^\bullet$ and $\partial P^\bullet$ be subcomplexes given as follows:
\[P^\bullet=\bigcup_{ \substack{  F \in \cS \\  F \cap Q = \emptyset}} F,\ 
\partial P^\bullet=\bigcup_{ \substack{  F \in \cS, \ F\subset\partial P\\  F \cap Q = \emptyset}} F.\]
Now, the quantity on the left in the statement is the relative Euler characteristic $\chi(P^\bullet,\partial P^\bullet)$.  If $Q=\emptyset$, this becomes $\chi(P,\partial P)=\chi(B^n,S^{n-1})$ in which case the theorem holds.

Now suppose that $Q\neq \emptyset$.  It suffices to show that the inclusion $\partial P^\bullet\hookrightarrow P^\bullet$ induces an isomorphism in homology.  Consider the commutative diagram
\[\xymatrix{
\partial P^\bullet \ar@{^(->}[r]\ar@{^(->}[d]&P^\bullet\ar@{^(->}[d]\\
\partial P\setminus Q\ar@{^(->}[r]&P\setminus Q 
}\]
The inclusion $P^\bullet\hookrightarrow P\setminus Q$ is a homotopy equivalence.  We give its homotopy inverse.  Let 
\[P^\bullet_i=P^\bullet\cup\bigcup_{ \substack{  F \in \cS \\  \dim(F)\leq i,\ F \cap Q \neq \emptyset}} (F\setminus Q).\]
Let $r_i:P^\bullet_{i+1}\rightarrow P^\bullet_{i}$ be given as the identity on cells $F$ disjoint from $Q$ and given by $r_F:F\setminus Q\rightarrow\partial F\setminus Q$ for cells intersecting $Q$ where $r_F$ is projection away from some point $x_F$ in the relative interior of $F\cap Q$.  
We define $r:P\setminus Q\rightarrow P^\bullet$ to be the composition $r_0\circ\dots\circ r_d:P\setminus Q=P^\bullet_{d}\rightarrow P^\bullet_0=P^\bullet.$

The inclusion $\partial P^\bullet\hookrightarrow \partial P\setminus Q$ is also a homotopy equivalence.  Its inverse is defined similarly to the map above.  
Finally, $\partial P\setminus Q\hookrightarrow P\setminus Q$ is a homotopy equivalence whose inverse can be given by projection from a point in the relative interior of $Q$.  Since these three maps induce isomorphisms in homology, so must $\partial P^\bullet\hookrightarrow P^\bullet$.
\end{proof}

The following lemma is analogous to \cite[Theorem~3.6]{KatzStapledon}. 

\begin{lemma} \label{l:independence}
The expression $\psi_{(X^\circ,\Sigma,\{0\})}$ above is independent of the choice of rational polyhedral structure $\Sigma$ on $\Trop(X^\circ)$.
\end{lemma}

\begin{proof}
Suppose $\Sigma'$ is a rational polyhedral structure on $\Trop(X^\circ)$ corresponding to a toric scheme $\P(\Sigma')_\O$.  After taking a common refinement, we can suppose that $\Sigma'$ is a refinement of $\Sigma$. This induces a proper morphism of toric schemes 
$\P(\Sigma')_\O \rightarrow \P(\Sigma)_\O$. If $\gamma'$ is a cell of $\Sigma'$, then the relative interior of $\gamma'$ lies in the relative interior $\Int(\gamma)$ of a unique cell $\gamma$ of $\Sigma$. By standard toric geometry \cite[Sec 2.1]{FulIntroduction},  the corresponding morphism of tori $U_{\gamma'} \rightarrow U_\gamma$ factors as 
\[
U_{\gamma'} \cong U_\gamma \times (\C^*)^{\dim \gamma - \dim \gamma'} \rightarrow U_\gamma,
\]
where the second map is projection onto the first coordinate. By \cite[Proposition~7.6]{LuxtonQu}, $X^\circ_{\gamma'}$ is the pullback of $X^\circ_\gamma$, and hence $[X^\circ_{\gamma'}] = [X^\circ_\gamma](\L - 1)^{\dim \gamma - \dim \gamma'}$. We compute 
\[
\sum_{\substack{\gamma'\in\Sigma'\\\gamma' \operatorname{bounded}}}\ \, [X^\circ_{\gamma'}] \, (1 - \L)^{\dim \gamma'}
 =  \sum_{\gamma \in \Sigma} [X^\circ_\gamma]  (1 - \L)^{\dim \gamma}  \left(\sum_{\substack{\gamma' \in \Sigma', \gamma' \operatorname{bounded}\\ \Int(\gamma') \subseteq \Int(\gamma) }} (-1)^{\dim \gamma - \dim \gamma'}\right).
\]
Hence, it is enough to show that
\[
 \sum_{\substack{\gamma' \in \Sigma', \gamma' \operatorname{bounded}\\ \Int(\gamma') \subseteq \Int(\gamma) }} (-1)^{\dim \gamma'}  =  \left\{\begin{array}{cl} (-1)^{\dim \gamma} & \text{if } \gamma \text{ is bounded } \\ 0 & \text{otherwise}. \end{array}\right.
\]
This follows directly from Lemma~\ref{l:euler} if we do the following:  let $C_\gamma$ be the cone over 
$\gamma \times 1$ in $\R^n \times \R$; choose $H$ to be an affine hyperplane such that $P=C_\gamma\cap H$ is a polytope not containing the origin; set $Q$ to be the intersection of $P$ with $\R^n\times \{ 0 \}$; and
let $\cS$ be polyhedral subdivision of $P$ induced by  the fan refinement of $C_\gamma$ induced by $\Sigma'$. %    $\Sigma' \cap |C_\gamma|$.
The cells in $\cS$ that intersect $Q$ correspond to unbounded cells in $\Sigma'$.
\end{proof}

Steenbrink has applied a result \cite[Theorem 5]{SteenbrinkMilnor} similar to Lemma \ref{l:independence} to study motivic Milnor fibres of function germs on toric singularities.

Since $\psi_{(X^\circ,\Sigma,\{0\})}$ is independent of the choice of $\Sigma$ by Lemma~\ref{l:independence}, after possible ramified base-extension of $\K$, it follows from \cite[Proposition~2.3]{HelmKatz} that we may choose $\Sigma$ such that $\P(\Delta)_\K$ is smooth. In this case, we may invoke the following result.

\begin{theorem}\cite[Theorem~5.1]{KatzStapledon}\label{t:old}
With the notation above, if $\P(\Delta)_\K$ is smooth, then the motivic nearby fiber $\psi_{X_{\Delta}}$ of $X_{\Delta}$ is equal to $\psi_{(X^\circ,\Sigma,\Delta)}$. 
\end{theorem}

By Theorem~\ref{t:old}, \eqref{e:summotivic}, and induction on dimension, we have  
\[
\psi_{X_{\Delta}} = \psi_{(X^\circ,\Sigma,\Delta)} = \sum_{\tau \in \Delta} \psi_{(X_\tau^\circ,\Sigma_\tau,\{0\})} =  \psi_{(X^\circ,\Sigma,\{0\})} + \sum_{\{0\} \neq \tau \in \Delta}  \psi_{X_\tau^\circ}.
\] 
Since $X_{\Delta} = \cup_{\tau \in \Delta} X^\circ_{\tau}$, the additivity of the motivic nearby fiber and the above expression  imply that 
\[
\psi_{X^\circ} = \psi_{X_{\Delta}} - \sum_{\{0\} \neq \tau \in \Delta}  \psi_{X_\tau^\circ} = \psi_{(X^\circ,\Sigma,\{0\})}. 
\]
This completes the proof of \eqref{e:open} and hence Theorem~\ref{t:comp}.  

Using \eqref{e:summotivic}, we immediately deduce the following corollary. 

\begin{corollary}\label{c:partial}
Let $X^\circ \subseteq (\K^*)^n$ be sch\"{o}n, and let $\Sigma$ be a rational polyhedral structure on $\Trop(X^\circ)$ extending to a polyhedral subdivision of $\R^n$. For each cell $\gamma$ of $\Sigma$, let 
$\rec(\gamma)$ denote the corresponding recession cone. For any subfan $\Delta'$ of the recession fan $\Delta$ of $\Sigma$, let 
$X_{\Delta'}$ denote the closure of $X^\circ$ in the corresponding toric variety $\P(\Delta')_{\K}$. Then the  motivic nearby fiber $\psi_{X_{\Delta'}}$ is given by:
\[
\psi_{X_{\Delta'}} =   \sum_{\substack{\gamma\in\Sigma\\\ \rec(\gamma) \in \Delta'}}\, [X^\circ_\gamma](1 - \L)^{\dim \gamma-\dim(\rec(\gamma))}.
\] 
\end{corollary}

\begin{remark}\label{r:generalize}
Note that when $\Delta' = \{0\}$ above, we recover Theorem~\ref{t:comp}, while when $\Delta' = \Delta$, then we recover the statement of  Theorem~\ref{t:old} without the assumption that   
$\P(\Delta)_\K$ is smooth. In this way, we see that Theorem~\ref{t:comp} is a generalization of Theorem~\ref{t:old}. Note that $X_{\Delta}$ is proper, and, while the assumption that $\P(\Delta)_{\K}$  
is smooth forces $X_{\Delta}$ to be smooth, in general, $X_{\Delta}$ and $\P(\Delta)_{\K}$ may have singularities. 
\end{remark}

\begin{remark}
Note that Theorem~\ref{t:comp} implies that if $X^\circ$ is sch\"on, then the expression $\psi_{(X^\circ,\Sigma,\{0\})}$ is not only independent of $\Sigma$ (Lemma~\ref{l:independence}), but independent on the choice of embedding $X^\circ \subseteq (\K^*)^n$. We do not know a direct proof of this fact.
\end{remark}

\begin{remark}
As in \cite[Section~3]{KatzStapledon}, the definition of  $\psi_{(X^\circ,\Sigma,\{0\})}$ can be extended to the case when $X^\circ$ is not necessarily sch\"on, but  the pair $(X^\circ,\P(\Sigma)_\O)$ is \emph{tropical}. In this case, the proof of Lemma~\ref{l:independence} holds unchanged and $\psi_{(X^\circ,\Sigma,\{0\})}$ is independent of the choice of $\Sigma$. The expression $\psi_{(X^\circ,\Sigma,\Delta)}$ was called the \emph{tropical motivic nearby fiber} in \cite{KatzStapledon}.  However, one can not expect an analogue of Theorem~\ref{t:comp}, as the following example demonstrates. 

Let $G(X,Y,Z)$ be a homogeneous polynomial over $\C$ of degree $3$ whose zero locus $V(G)$ in $\P^2$ is a nodal cubic curve.  Suppose further that $G$ has the following properties
\begin{enumerate}
\item all coefficients of degree $3$ monomials in $G$ are non-zero,
\item the node of $V(G)$ lies in $(\C^*)^2\subset\P^2$, and
\item $V(G)$ intersects each coordinate lines in $3$ distinct points.
\end{enumerate}
It is possible to find such a $G$ by applying a generic element of $\operatorname{Gl}_3(\C)$ to the equation of a nodal cubic.
The tropicalization of $V(G)^\circ \subseteq (\C^*)^2$ in $\R^2$ consists of the origin and three rays in the directions $(1,0),(0,1),(-1,-1)$, each with multiplicity $3$.  Now let $H$ be a generic homogeneous polynomial of degree $3$.  Consider $F=G+tH$ considered as a homogeneous polynomial over $\K$.  Now, $V(F)$ is a smooth cubic over $\K$.  Consequently, for $t\neq 0$ sufficiently small, $V(F)_{\gen}^\circ$ is a smooth cubic over $\C$ with $9$ points removed.  By construction, the tropicalization of $V(F)^\circ \subseteq (\K^*)^2$ is the same as that of $V(G)^\circ$.  %Moreover, the standard polyhedral structure on the tropicalization is adapted for $V(F)$.  
Moreover, $(V(F)^\circ,\P(\Sigma)_\O)$ is tropical, where $\Sigma$ is the standard polyhedral structure on the tropicalization of $V(F)^\circ$.  
There is a single bounded cell which is the origin.   Since $\init_0 V(F)^\circ=V(G)^\circ$, we have $\psi_{(V(F)^\circ,\Sigma,\{0\})} = [V(G)^\circ]$. 
%open tropical motivic nearby fiber of $V(F)\rightarrow \D^*$ is $[V(G)^\circ]$.  
This is a nodal cubic minus the $9$ points of intersection with the coordinate lines.  
Because a nodal cubic is isomorphic to a projective line with two points identified, we have $e(V(G)^\circ) = -8 \ne e(V(F)_{\gen}^\circ) = -9$ violating the final statement in Corollary~\ref{c:main}. 
\end{remark}

\section{The motivic nearby fiber and limit mixed Hodge structures}\label{s:motivic}

As observed in the introduction, by composing the motivic nearby fiber \eqref{e:mmap} with a motivic invariant over $\C$, we obtain a motivic invariant over $\K$ to which we can apply Theorem~\ref{t:comp}. 
In this section, we introduce some known results from the theory of  limit mixed Hodge structures. We recommend  \cite{Peters}  and \cite{PSMixed} as references. 
The theory was developed by many authors including Deligne, Katz, %Landman~\cite{LanPicard}, 
Clemens~\cite{CleDegeneration}, Schmid~\cite{SchVariation}, 
Steenbrink~\cite{SteLimits}, and Saito \cite{Saito}.

Throughout this section, if a complex vector space $B$ admits a mixed Hodge structure \cite{PSMixed} with corresponding decomposition
%$\{ H^{p,q}(B) \}_{p,q}$, 
\[
B \cong \bigoplus_{p,q} H^{p,q}(B), 
\] 
then we write 
$h^{p,q}(B) = h^{q,p}(B) = \dim  H^{p,q}(B)$. For a sequence of such vector spaces $B_\bullet = \{ B_m \mid m \ge 0 \}$, set 
$e^{p,q}(B_\bullet) = e^{q,p}(B_\bullet) = \sum_{m} (-1)^m h^{p,q}(B_m)$.
Then the \emph{Hodge polynomial} of $B_\bullet$ is defined by
\[
E(B_\bullet) = E(B_\bullet; u,v) = \sum_{p,q} e^{p,q}(B_\bullet) u^p v^q. 
\]

\subsection{Motivic invariants over $\C$}

In  \cite{DelTheory}, Deligne proved that the  $m^{\textrm{th}}$ cohomology group  with compact supports $H^m_c  (V)$ of a complex variety $V$
admits a canonical mixed Hodge structure with decreasing filtration $F^\bullet$ called the \define{Hodge filtration} and increasing filtration $W_\bullet$ called the 
\define{Deligne weight filtration}. The set of numbers $\{ h^{p,q}(H_c^m(V)) \}_{p,q,m}$ are called the \define{mixed Hodge numbers} of $V$, and the corresponding Hodge polynomial $E(V;u,v)$ is the 
\define{Hodge-Deligne polynomial} of $V$. 
The corresponding motivic invariant over $\C$ is the \emph{Hodge-Deligne map}
\[
E:K_0(\Var_\C)\rightarrow\Z[u,v],  \: [V] \mapsto E(V;u,v).
\]
%where $E(V;u,v)$ is the \define{Hodge-Deligne polynomial} of a complex variety $V$. 

The Hodge-Deligne polynomial of $V$ specializes to the \define{$\chi_y$-characteristic} $E(V;u,1)$ of $V$. Its coefficients are alternating sums of the dimensions of the graded pieces of the Hodge filtration on 
the cohomology of $V$ with compact supports. The \define{Euler characteristic} $e(V)$ is obtained via the specialization $e(V) = E(V;1,1)$.

\begin{example}\label{e:torus}
Recall that we write $\L := [\A^1]$  in the Grothendieck ring $K_0(\Var_k)$ for any field $k$. The complex affine line has $h^{1,1}(H^2_c(\A^1)) = 1$ and all other mixed Hodge numbers equal to zero. Hence its Hodge-Deligne polynomial is
$E(\A^1) = uv$.

The $n$-dimensional torus $(\C^*)^n$ has $[(\C^*)^n] = (\L - 1)^n$ in $K_0(\Var_\C)$, $E((\C^*)^n) = (uv - 1)^n$, and 
its mixed Hodge numbers are all zero except 
\[ h^{k,k}(H_c^{n+k}((\C^*)^n)) = \binom{n}{k},\]
for all $k$.
\excise{
Consider the $n$-dimensional torus $(\C^*)^n$. In the Grothendieck ring $K_0(\Var_\C)$, 
\[ [(\C^*)^n] = [(\A^1 - \{\pt \})^n] = (\L - 1)^n,\]
and hence $E((\C^*)^n) = (uv - 1)^n$. Indeed, the mixed Hodge numbers of $Z$ are all zero except 
\[ h^{k,k}(H_c^{n+k}((\C^*)^n)) = \binom{n}{k},\]
for all $k$.
}
\end{example}

\excise{
\begin{remark}
In fact, the Hodge-Deligne map $E$ factors as 
\[ \xymatrix{K_0(\Var_\C)\ar[r]^{\chi_{\Hdg}}&K_0(\HS)\ar[r]^P&\Z[u,v]},\]
where the Hodge characteristic $\chi_{\Hdg}:K_0(\Var_\C)\rightarrow K_0(\HS)$ takes
a complex $V$ to the virtual Hodge structure $\sum_k (-1)^k H^k(V)$, and the Hodge polynomial map $P:K_0(\HS)\rightarrow\Z[u,v]$ takes a Hodge structure to its Hodge polynomial \cite[Ch. 2]{Peters}.
\end{remark}
}

\subsection{Motivic invariants over $\K$}\label{s:motivicfamily}

Recall from the introduction that we regard a variety $X$ over $\K$ as a family of complex varieties over the disc $\D^*$, and we fix a non-zero fiber $X_{\gen} := f^{-1}(t)$ for some $t \in \D^*$.  Then the cohomology
groups $H_c^m(X_{\gen})$ admit a weight filtration $M_\bullet$ called the \define{monodromy weight filtration}. %We will recall the geometric description of this filtration in the next subsection. 
We will write $H_c^m(X_{\infty})$ to denote $H_c^m(X_{\gen})$ with the mixed Hodge structure $(F^\bullet,M_\bullet)$.
The corresponding mixed Hodge numbers are denoted 
$h^{p,q}(H_c^m(X_{\infty}))$ and are called the \define{limit mixed Hodge numbers} of $X$. The corresponding Hodge polynomial is denoted  $E(X_\infty;u,v)$, and is called the  \define{limit Hodge-Deligne polynomial} of $X$. It is the composition of the motivic nearby fiber and the Hodge-Deligne map i.e. $E(X_\infty;u,v) = E(\psi_X)$. 
 It specializes to both the $\chi_y$-characteristic of $X_{\gen}$ 
and the Euler characteristic of $X_{\gen}$ (see Remark~\ref{r:specializelimit} below).

The monodromy weight filtration  encodes the action of the logarithm of the monodromy operator on the $W_\bullet$-graded pieces of $H_c^m(X_{\gen})$. We explain this statement in detail below. 

The cohomology groups $H^m_c(X_t)$  of the fibers are isomorphic as vector spaces but have a Hodge structure which varies.  Because $H^m_c(X_t)$ forms a locally trivial fiber bundle, parallel transport gives a monodromy transformation, $T: H_c^m (X_{\gen}) \rightarrow H_c^m (X_{\gen})$.  It turns out that $T$ is quasi-unipotent, that is, some multiple of $T$ is unipotent.  
After replacing $T$ by some power which corresponds to pulling back the family by a map $\D^*\rightarrow \D^*$ ramified over the puncture, we may take its logarithm, $N=\log T$, and obtain a nilpotent operator.  Moreover, the monodromy map $T$ preserves the weight filtration $W_\bullet$,   
and $N: H_c^m(X_{\infty}) \rightarrow H_c^m(X_{\infty})$ is a morphism of mixed Hodge structures of type $(-1,-1)$.

It follows that for every non-negative integer $r$, $N$ restricts to a nilpotent operator $N(r)$ on the graded piece $\Gr_r^W H_c^m (X_{\gen})$. Let $F(r)^\bullet$ and $M(r)_\bullet$ denote the filtrations on $\Gr_r^W H_c^m (X_{\gen})$ induced by $F^\bullet$ and $M_\bullet$ respectively. Then $N(r)^{r + 1} = 0$ and $M(r)_\bullet$ is the filtration obtained from $N(r)$ which determines and is determined by the Jordan block decomposition of $N(r)$.
Indeed, we may inductively define a unique increasing filtration 
\[
0 \subseteq M(r)_0 \subseteq M(r)_1 \subseteq \cdots \subseteq M(r)_{2r} = \Gr_r^W H_c^m (X_{\gen})
\]
satisfying the following properties for any non-negative integer $k$,
\begin{enumerate}
\item $N(r)( M(r)_k ) \subseteq M(r)_{k - 2}$,
\item\label{e:hard} the induced map $N(r)^k: \Gr^{M(r)}_{r + k} \Gr_r^W H_c^m (X_{\gen}) \rightarrow \Gr^{M(r)}_{r - k} \Gr_r^W H_c^m (X_{\gen})$ is an isomorphism. 
\end{enumerate}
%For example,  $M(r)_0 = \im N(r)^r$ and $M(r)_{2r - 1} = \ker N(r)^r$.  
The pair $(F(r)^\bullet,M(r)_\bullet)$ determine a limit mixed Hodge structure on $\Gr_r^W H_c^m (X_{\gen})$. Moreover, $N(r)$ is a morphism of mixed Hodge structures of type $(-1,-1)$.
We will write  $\Gr_r^W H_c^m (X_\infty)$ when referring to  $\Gr_r^W H_c^m (X_{\gen})$ with the limit mixed Hodge structure. We denote the corresponding mixed Hodge numbers by 
$h^{p,q,r}(H_c^m(X_{\infty}))$ and call them the \define{refined limit mixed Hodge numbers} of $X$.  
For each $r$, we encode the corresponding Hodge polynomial as the coefficient of $w^r$ in a polynomial $E(X_\infty;u,v,w)$ that we will call the \define{refined limit Hodge-Deligne polynomial}. 

\begin{remark}\label{r:natural}
It follows from Saito's theory of mixed Hodge modules \cite{Saito} (see \cite{Arapura} for a survey) that most natural morphisms between varieties $X$ over $\K$ give rise to morphisms of complex varieties $X_{\gen}$ that respect the filtrations $(F^\bullet,W_\bullet,M_\bullet)$. In particular, if $U \subseteq X$ is an open inclusion and $V = X \smallsetminus U$, then  the corresponding long exact sequence of cohomology with compact supports for the triple $(X_{\gen},U_{\gen},V_{\gen})$ consists of morphisms that preserve the Hodge filtration and both the Deligne and monodromy weight filtrations (c.f. proof of \cite[Lemma~14.61]{PSMixed}; see also \cite{FFVar} for the classical approach). In particular, it follows from Remark~\ref{r:specialize} that the refined limit Hodge-Deligne polynomial is a motivic invariant over $\K$. That is, we may consider the \emph{refined Hodge-Deligne map}
\[
\bar{E}:K_0(\Var_\K)\rightarrow\Z[u,v,w],  \: [X] \mapsto E(X_\infty;u,v,w).
\]
\end{remark}

\subsection{Properties of the refined limit Hodge-Deligne polynomial} We now collect some of the basic properties of the above motivic invariants over $\K$. In particular, we will give a complete description
of the following commutative diagram of ring homomorphisms 
\[\xymatrix{
K_0(\Var_\K) \ar[r]^{\bar{E}} \ar[d]^\psi  &\Z[u,v,w]  \ar[r]^{\substack{u \mapsto uw^{-1} \\ v \mapsto 1}}   \ar[d]^{w \mapsto 1}  & \Z[u,w] \ar[d]^{w \mapsto 1}  &\\
K_0(\Var_\C) \ar[r]^{E} &\Z[u,v] \ar[r]^{v \mapsto 1} & \Z[u]  \ar[r]^{u \mapsto 1}  & \Z.
}\] 

\begin{remark}
The refined limit mixed Hodge numbers have the following explicit description in terms of $(F^\bullet, W_\bullet, M_\bullet)$,
\[h^{p,q,r}(H_c^m(X_{\infty})) =\dim(\Gr_{F(r)}^p\Gr^{M(r)}_{p+q}\Gr_r^W H^m(X_\infty)).\]
If we sum over $r$, we discard the refinement by the weight filtration and obtain the following relation with the limit mixed Hodge numbers,
\begin{equation}\label{e:limitnumber}
 h^{p,q}(H_c^m(X_\infty))= \dim(\Gr_F^p\Gr^{M}_{p+q} H^m(X_\infty)) = \sum_r h^{p,q,r}(H_c^m(X_{\infty})).
\end{equation}
Similarly, if we sum over $q$,  we discard the monodromy filtration refinement and obtain mixed Hodge numbers of $X_{\gen}$,
\begin{equation}\label{e:number}
h^{p,r-p}(H_c^m(X_{\gen})) = \sum_q h^{p,q,r}(H_c^m(X_{\infty})). 
\end{equation}
Summing the refined limit mixed Hodge numbers over $q$ and $r$ gives the following relation between the limit mixed Hodge numbers and the mixed Hodge numbers of $X_{\gen}$,
\begin{equation}\label{e:relate}
\sum_q h^{p,q}(H_c^m(X_\infty)) = \sum_q h^{p,q}(H_c^m(X_{\gen})).
\end{equation}
\end{remark}
\begin{remark}\label{r:specialize}
The refined limit Hodge-Deligne polynomial  is described explicitly in terms of the refined limit mixed Hodge numbers as 
\[
E(X_\infty;u,v,w) =  \sum_{p,q,r} e^{p,q,r}(X_\infty) u^p v^q w^r,
\]
where 
\[
e^{p,q,r}(X_\infty)   = e^{q,p,r}(X_\infty)   = \sum_m (-1)^m h^{p,q,r}(H_c^m(X_{\infty})). 
\]
Using \eqref{e:limitnumber} and \eqref{e:number}, we see that the refined limit Hodge-Deligne polynomial specializes to both the 
limit Hodge-Deligne polynomial and the Hodge-Deligne polynomial of $X_{\gen}$,
\[
E(X_\infty;u,v) = E(X_\infty;u,v,1),
\]
\[
E(X_{\gen};u,w) = E(X_\infty;uw^{-1},1,w).
\]
\end{remark}

\begin{remark}\label{r:specializelimit}
We see from Remark~\ref{r:specialize} that the limit Hodge-Deligne polynomial specializes to both the $\chi_y$-characteristic of $X_{\gen}$,
\begin{equation}\label{e:xy} E(X_\infty;u,1) = E(X_{\gen};u,1),\end{equation}
and the Euler characteristic of $X_{\gen}$
\[
E(X_\infty;1,1) = e(X_{\gen}), 
\]
\end{remark}

\begin{remark}\label{r:unimodal}
With the notation above, since $N(r)$ is a morphism of mixed Hodge structures of type $(-1,-1)$, the isomorphisms \eqref{e:hard} imply that each vertical strip of the Hodge diamond of
$\Gr_r^W H_c^m (X_\infty)$ is a symmetric, unimodal sequence of non-negative integers. That is, for $0 \le k \le r$, the sequence 
$\{ h^{k + i,i,r}(H_c^m(X_{\infty})) \mid 0 \le i \le r -k \}$ is symmetric and unimodal.
\end{remark}

\begin{remark}\label{r:symmetriesrefined}
By construction, the refined limit mixed Hodge numbers are symmetric in $p$ and $q$ i.e.  $h^{p,q,r}(H_c^m(X_{\infty})) = h^{q,p,r}(H_c^m(X_{\infty}))$. It follows from 
Remark~\ref{r:unimodal} that they satisfy the additional symmetry:
\[
h^{p,q,r}(H_c^m(X_{\infty})) = h^{r-p,r-q,r}(H_c^m(X_{\infty})).
\]
In particular, the refined limit Hodge-Deligne polynomial satisfies the symmetries
\[
E(X_{\infty};u,v,w) = E(X_{\infty};v,u,w),
\]
\[
E(X_{\infty};u,v,w) = E(X_{\infty};u^{-1},v^{-1},uvw).
\]
\end{remark}

\begin{example}\label{e:trivial2}
If $V$ is a complex variety and $X = V \times_\C \K$, then $X$
%If $X$ is defined over $\C$, then it 
may be regarded as a trivial family over $\D^*$. In this case, $N$ is identically zero, $M_\bullet$ coincides with the Deligne filtration $W_\bullet$, and $E(X_\infty;u,v) = E(V;u,v)$. 
Moreover, \[
E(X_\infty;u,v,w)  = E(V;uw,vw).
\]
\end{example}

\begin{example}
If $X$ is smooth and proper, then $\Gr_r^W H^m (X_{\gen}) = 0$ unless $m = r$. In this case, the monodromy weight filtration encodes the Jordan block decomposition of $N: H^m(X_{\infty}) \rightarrow H^m(X_{\infty})$. Moreover, 
\[
E(X_\infty;u,v,w) =   \sum_{p,q,m} (-1)^m h^{p,q}(H^m(X_{\infty})) u^p v^q w^m.
\]
\end{example}

\begin{remark}\label{r:zero} %If there is only one non-zero Hodge number of $H_c^m(X_{\gen})$ for each $m$, then $N$ must vanish identically as there is no non-trivial domain or range for $N$ considered as a morphism of Hodge structures of type $(-1,-1)$.  Indeed, 
We claim that if there exists a function $\nu: \Z \rightarrow \Z$, such that $H^{p,q}(H_c^m(X_{\gen})) = 0$ for $p \ne \nu(m)$, then $N = 0$.
Indeed, by \eqref{e:relate},
\[
\sum_q h^{p,q}(H_c^m(X_\infty)) = \sum_q h^{p,q}(H_c^m(X_{\gen})) = 0 
\] 
for $p \neq \nu(m)$. Since $N: H_c^m(X_{\infty}) \rightarrow H_c^m(X_{\infty})$ is a morphism of mixed Hodge structures of type $(-1,-1)$, for any $(p,q)$,
either the source or target of the induced map $N: H^{p,q}(H_c^m(X_\infty)) \rightarrow H^{p-1,q-1}(H_c^m(X_\infty))$ is zero. 
\end{remark}

\begin{example}\label{e:torusfiber}
By Example~\ref{e:torus} and Remark~\ref{r:zero}, if $X_{\gen} \cong (\C^*)^n$, then $N = 0$. Hence, by Example~\ref{e:trivial2},
\[
E(X_\infty;u,v,w) = E((\K^*)^n;u,v,w) =  E((\C^*)^n;uw,vw) = (uvw^2 - 1)^n.
\]
\end{example}

\section{Subdivisions of lattice polytopes}\label{s:combinatorial}

In this section we gather together the relevant facts that we will need about the combinatorics of subdivisions of polytopes. Details and proofs of all statements can be found in \cite{otherpaper}. 
We say that the empty polytope has dimension $-1$. 

A \define{polyhedral subdivision} of a polytope $P\subset\R^n$ is a subdivision of $P$ into a finite number of polytopes such that the intersection of any two polytopes is a (possibly empty) face of both.  A \define{lattice polyhedral subdivision} of a lattice polytope $P$ is a polyhedral subdivision of $P$ into lattice polytopes.  A natural class of polyhedral subdivisions are the regular subdivisions.  They are induced by a height function $\omega:P\cap\Z^n\rightarrow \R$.  The cells of the subdivision are the projections 
of the bounded faces of the convex hull of $\operatorname{UH}=\{ (u, \lambda) \mid \lambda \geq\omega(u) \} \in \R^n \times \R$.  A subdivision is said to be \define{regular} if it is induced by some height function.  For more details, see \cite{Triangulations,GKZ}.
 
We first recall some definitions concerning the combinatorics of Eulerian posets.  
Consider a finite poset  $B$  containing a minimal element $\hat{0}$ and a maximal element $\hat{1}$. For any pair $z \leq x$ in $B$, we can consider the interval $[z,x] = \{ y \in B \mid z \leq y \leq x \}$. 
Assume that $B$ is \emph{graded} in the sense that for every $x \in B$, every maximal chain in 
 the interval $[\hat{0},x]$ has the same length $\rho(x)$. We call $\rho: B \rightarrow \N$ the \emph{rank function} of $B$, and call $\rho(\hat{1})$ the \emph{rank} of $B$. %We call $\rho: B \rightarrow \N$ the \emph{rank function} of $B$, and 
Then $B$ is \define{Eulerian} if every interval $[z,x]$ with $z < x$ has as many elements of odd rank as even rank.

\begin{example}\label{e:polytope}
The poset of faces of a polytope $P$ (including the empty face) is an Eulerian poset under inclusion. 
%\emph{Throughout, we will consider the empty face to have dimension $-1$. }
Then $\rho(Q) = \dim Q + 1$, for any face $Q$ of $P$. 
%If $\cS$ is a polyhedral decomposition of $P$, then the poset of faces of $\cS$ (including the empty face) is a 
%lower Eulerian poset.  
Let $\cS$ be a polyhedral subdivision of $P$, and let $F$ be  a (possibly empty) cell of $\cS$. As a poset, the \define{link} $\lk_\cS(F)$  of $F$ in $\cS$ consists of all 
cells $F'$ of $\cS$ that contain $F$ under inclusion, and we have that the interval $[F,F']$ is an Eulerian poset. 
% is a lower Eulerian poset of rank $\dim P - \dim F$.  \emph{We let $\sigma(F)$ denote the smallest face of $P$ containing
%$F$}. If $F$ is the empty face of $\cS$, then $\sigma(F)$ is the empty face of $P$. 
\end{example}

\begin{example}
If $B$ is a poset, then $B^*$ is the poset with the same elements as $B$ and all orderings reversed. In particular, $B$ is Eulerian if and only if $B^*$ is Eulerian. 
\end{example}

The \define{$g$-polynomial} of an Eulerian poset is defined recursively and was introduced by Stanley \cite[Corollary~6.7]{StaSubdivisions}. 

\begin{definition}\label{d:g}
Let $B$ be an Eulerian poset of rank $n$. If $n = 0$, then $g(B;t) = 1$. If $n > 0$, then $g(B;t)$ is the unique polynomial
of degree strictly less the $n/2$ satisfying
\[
t^{n}g(B;t^{-1}) = \sum_{x \in B} (t - 1)^{n - \rho(x)} g([\hat{0},x];t). 
\]
\end{definition}

The following theorem of Stanley giving an inversion formula for $g$ will be useful in Section \ref{s:intcoh}:

\begin{theorem}   \label{t:Eulerianinverse}\cite[Corollary~8.3]{StaSubdivisions}
If $B$ is Eulerian and has positive rank, then 
\[
\sum_{x \in B} (-1)^{\rho(x)} g([\hat{0},x];t)g([x,\hat{1}]^*;t) = \sum_{x \in B} (-1)^{\rho(x)} g([\hat{0},x]^*;t)g([x,\hat{1}];t)  =  0.
\]
\end{theorem}

We will be interested in the following example \cite[Example~7.2]{StaSubdivisions}. 

\begin{example}\label{e:link}
Let $\cS$ be a polyhedral subdivision of a polytope $P$. If $F$ is a  (possibly empty) cell of $\cS$, then the \define{$h$-polynomial} of $\lk_\cS(F)$ is defined by
\[
t^{\dim P - \dim F}h(\lk_\cS(F);t^{-1}) = \sum_{\substack{F' \in \cS \\\ F \subseteq F'}} (t - 1)^{\dim P - \dim F'} g([F,F'];t). 
\]
%When $F$ is the empty face, we will write $h(\cS,t)  = h(\lk_\cS(F),t)$. 
\end{example}

We now recall some basic Ehrhart theory. 
 Let $P$ be a non-empty lattice polytope in a lattice $M$ of rank $n$. % of dimension $n$.  %After replacing $M$ with a smaller lattice, we may and will assume that $n = \dim P$. 
 For $m\in \Z_{> 0}$, consider the function $f_P(m)=\#(mP\cap M)$.  By Ehrhart's theorem  \cite[Section~3.3]{BRComputing}, $f_P(m)$ is a polynomial of degree $\dim P$, called
 the \define{Ehrhart polynomial} of $P$.  
 It follows that we can write
\begin{equation}\label{e:Ehrhartpoly}
f_P(m)=f_0(P)+f_1(P)m+\dots+f_{n}(P)m^{n}, 
\end{equation}
where $f_i(P)\in \Z$, and
\[
1 + \sum_{m > 0} f_P(m) u^m = \frac{h^*(P;u)}{(1 - u)^{\dim P + 1}}. 
\]
where  $h^*(P;u)$ of $P$ is a polynomial of degree at most $\dim P$ called the \define{$h^*$-polynomial} of $P$ (see, for example, \cite[Section~3.3]{BRComputing}). 
Note that if $P$ is empty, then we set $f_P(m) \equiv 0$ and  $h^*(P;u) = 1$.  We have
$h^*(P;1) = (\dim P)!\vol(P)$ where $\vol(P)$ is the Euclidean volume of $P$.

These invariants play the central role in the theory of valuations on polytopes. 
The definition given below is a priori weaker than the usual definition of valuations but is equivalent as a consequence of Lemma~\ref{l:simplex}.

\begin{definition}\label{d:valuation} Let $\cp_M$ be the set of lattice polytopes for a lattice $M$ and let $G$ be a group.  A $G$-valued \define{valuation} on $\cp_M$ is a map $\phi:\cp_{M}\rightarrow G$ satisfying
\begin{enumerate}
\item \label{v:c1} If $\cS$ is a regular lattice subdivision of $P$ with top dimensional cells $P_1,\dots,P_m$, $\phi$ satisfies the inclusion/exclusion relation
\[
\phi(P) = \sum_{ \substack{ F \in \cS \\ F \nsubseteq \partial P } } (-1)^{\dim P - \dim F} \phi(F),
\]

\item \label{v:c2} $\phi(\emptyset)=0$, and 
\item \label{v:c3} $\phi(P)=\phi(UP+u)$ for $P\in\cp_{M}$, $U\in\Aut(M)$, $u\in M$.
\end{enumerate}
\end{definition}

The lemma below is non-trivial since not every lattice polytope admits a  lattice polyhedral subdivision into unimodular simplices.  This lemma is an  adaptation of \cite[Prop 19.2]{Gruber} which is stated for general lattice subdivisions. 

\begin{lemma}\label{l:simplex}
Valuations are determined by their values on unimodular simplices: if $\phi_1,\phi_2$ are valuations that are equal on unimodular simplices, then $\phi_1=\phi_2$.
\end{lemma}
\begin{proof}  
Let $\cG$ be the free Abelian group generated by convex lattice polytopes in $M$.  Let $\cH$ be the subgroup generated by the following:
\begin{enumerate}
\item For $\cS$ is a regular subdivision of $P$, 
\[P - \sum_{ \substack{ F \in \cS \\ F \nsubseteq \partial P } } (-1)^{\dim P - \dim F} \phi(F),\]
\item $\emptyset,$
\item $P-(UP+u)$ for $P\in\cp_{\Z^n}$, $U\in\Aut(M)$, $u\in M$
\end{enumerate}
We show that $\cG/\cH$ is generated by unimodular simplices.  We induct on the dimension of $M$.  Let $\cH'$ be the subgroup of $\cG$ generated by convex lattice polytopes of $M$ whose affine span is not full-dimensional.  By induction, we may suppose $\cH'/(\cH'\cap\cH)$ is generated by unimodular simplices.  Now, it suffices to show that $\cG/(\cH+\cH')$ is generated by a $d$-dimensional unimodular simplex.

First, every polytope has a regular triangulation (see, for example, \cite[Proposition~2.2.4]{Triangulations}).  Therefore, every polytope in $\cG/(\cH+\cH')$ can be written as a formal sum of lattices simplices.  It remains to show that every lattice simplex can be written as a multiple of a unimodular simplex.  We induct on the volume of the lattice simplex.  Let $P\subset M_\R$ be a $d$-dimensional lattice simplex.  If $\vol(P)=1$ then we're done.  Suppose $\vol(P)=V\geq 2$, and let $F_0,\dots,F_d$ denote the facets of $P$. By the proof of \cite[Proposition~19.1]{Gruber}, there is a point $p\in M$ such that $\vol(\Conv(F_i\cup \{p\}))<V$.  Let $\omega:\operatorname{Vert}(P)\cup\{p\}\rightarrow\R$ be the height function that is $0$ on the vertices of $P$ and $1$ on $p$.  The graph of the height function lies in $M_\R\times\R$, and because the points in the graph are affinely independent, their convex hull is a simplex.  The projections of the convex hull by $\pi:M\times\R\rightarrow M$ is $\Conv(P\cup\{p\})$.  Moreover, the projections of the upper faces or the lower faces each give regular subdivisions $\cS_{\operatorname{upper}}$, $\cS_{\operatorname{lower}}$ of $\Conv(P\cup\{p\}$.
 The top-dimensional lower faces of the convex hull are $P$ and some faces $\Conv(F_i\cup\{p\})$ for $i\in I$ for some subset $I\subset\{0,\dots,n\}$.  The top-dimensional upper faces of the convex hull are $\Conv(F_i\cup\{p\})$ for $i\not\in I$.  The subdivision relation in $\cG/(\cH+\cH')$ gives
\[\Conv(P\cup \{p\})=P+\sum_{i\in I} \Conv(F_i\cup\{p\})=\sum_{i\not\in I} \Conv(F_i\cup \{p\}).\]
Therefore, we have in $\cG/(\cH+\cH')$,
\[P=\sum_{i\in I}  \Conv(F_i\cup\{p\})-\sum_{i\in I}  \Conv(F_i\cup\{p\}).\]
This gives an expression for $P$ as a formal sum of simplices of smaller volume.
\end{proof}

Since $P\mapsto f_P(m)$ is a $\Z[m]$-valued valuation, each $f_i: P \mapsto f_i(P)$ is a $\Z$-valued valuation for $i = 0,\ldots, n$. For example,
\[f_0(P)=\begin{cases} 1 \text{ if } P\neq \emptyset\\
0 \text{ if } P=\emptyset.
\end{cases} \]
By the Betke-Kneser theorem (\cite{BetkeKneser}, \cite[Theorem~19.6]{Gruber}), $\{f_0,\dots,f_n\}$ are a $\Z$-basis for the group of all $\Z$-valued valuations of $\cp_M$.

\begin{example}\label{e:aval}
For a fixed non-negative integer $m$, we have a valuation
\[
P \mapsto f_P(m) = f_0(P)+f_1(P)m+\dots+f_{n}(P)m^{n}. 
\]
\end{example}

The \define{local $h^*$-polynomial} $\lc(P;u)$ of $P$ was introduced by Stanley in \cite[Example~7.13]{StaSubdivisions}, generalizing the definition of Betke and McMullen in the case of a simplex \cite{BMLattice},
and was independently introduced by Borisov and Mavlyutov in \cite{BMString}, 
\begin{equation*}%\label{e:localdef}
\lc(P; u) =  \sum_{Q \subseteq P} (-1)^{\dim P - \dim Q}h^*(Q;u)g([Q,P]^*;u). 
\end{equation*}

Our main combinatorial invariants are introduced below, and first appeared in \cite[Sections~7-9]{otherpaper}. 
An explicit geometric description of these invariants is provided in Corollary~\ref{c:explicit}. 

\begin{definition}\label{d:new}
Let $\cS$ be a lattice polyhedral subdivision  of a lattice polytope $P$. Then the \define{limit mixed $h^*$-polynomial} of $(P,\cS)$ is
\[
h^*(P,\cS;u,v) := \sum_{F \in \cS} v^{\dim F + 1}\lc(F; uv^{-1})  h(\lk_\cS(F);uv).
\]
The \define{local limit mixed $h^*$-polynomial} of $(P,\cS)$ is
\[
\lc(P, \cS; u,v) :=  \sum_{Q \subseteq P} (-1)^{\dim P - \dim Q}h^*(Q, \cS|_Q;u,v) g([Q,P]^*;uv).
\]
The \define{refined limit mixed $h^*$-polynomial} of $(P,\cS)$ is
\[
h^*(P,\cS;u,v,w) = \sum_{Q \subseteq P} w^{\dim Q + 1}\lc(Q, \cS|_Q;u,v) g([Q,P];uvw^2).
\]
If $\cS$ is the trivial subdivision of $P$, with cells of $\cS$ given by the faces of $P$,  then
we write $h^*(P;u,v) = h^*(P,\cS;u,v)$ and call the polynomial the
\define{mixed $h^*$-polynomial}.
If $P$ is empty, then $h^*(P,\cS;u,v,w) = h^*(P,\cS;u,v) =  \lc(P, \cS;u,v)  = 1$.  
\end{definition}

%%%%%%%%%%%%%%%%%%%%%%%%%%%%%%%%%

%%%%%%%%%%%%%%%%%%%%%%%%%%%%%%%%%

%The invariants above satisfy the following properties that are proved in \cite{otherpaper}. 

The following theorem is proved in \cite[Theorem~9.2]{otherpaper}. % and \cite[Lemma~9.6]{otherpaper}. 

\begin{theorem}\label{t:combinatorics}
Let $\cS$ be a lattice polyhedral subdivision  of a lattice polytope $P$. 
Then the  refined limit mixed $h^*$-polynomial satisfies the following properties: 

\begin{enumerate}

\item\label{prop'} The refined limit mixed $h^*$-polynomial is invariant under the interchange of $u$ and $v$, and satisfies the additional symmetry 
\[
h^*(P,\cS,u,v,w) = h^*(P,\cS,u^{-1},v^{-1},uvw).
\]

\item\label{prop1} The refined limit mixed $h^*$-polynomial specializes to the limit mixed $h^*$-polynomial
\[
h^*(P,\cS;u,v,1) = h^*(P,\cS;u,v).   
\]

\item\label{prop2} The refined limit mixed $h^*$-polynomial specializes to the mixed $h^*$-polynomial
\[
h^*(P,\cS;uw^{-1},1,w) = h^*(P;u,w).   
\]

\item\label{prop3} The refined limit mixed $h^*$-polynomial specializes to the $h^*$-polynomial
\[
h^*(P,\cS;u,1,1) = h^*(P;u).   
\]

\item\label{prop4} The degree of $h^*(P,\cS;u,v,w)$ as a polynomial in $w$ is at most $\dim P + 1$.  Moreover, the coefficient of 
$w^{\dim P + 1}$ is the local limit mixed $h^*$-polynomial
$\lc(P, \cS; u,v)$. 

%
%\item\label{prop5} %The degree of $h^*(P,u,w)$ as  a polynomial in $w$ is at most $\dim P$. 
%Every monomial $u^pw^q$ appearing in $h^*(P;u,w)$ satisfies $p + q \le \dim P + 1$. Moreover, the coefficient of $u^p w^{\dim P + 1 - p}$ in 
%$h^*(P;u,w)$ is equal to the coefficient of $u^p$ in the local $h^*$-polynomial $\lc(P; u)$. 

\item\label{prop6} The limit mixed $h^*$-polynomial can be written in terms of  mixed $h^*$-polynomials,
\[ 
h^*(P, \cS;u,v) = \sum_{ \substack{ F \in \cS \\ F \nsubseteq \partial P } } (uv - 1)^{\dim P - \dim F}  h^*(F;u,v), 
\]
where $\partial P$ denotes the boundary of $P$. %$\sigma(F)$ denotes the smallest (possibly empty) face of $P$ containing $F$. 

%\item\label{prop7} We have the following symmetry
%\[
%\Lambda(P,\cS,\Delta_P';u,v,w) = (uvw^2)^{\dim P + 1}\Lambda(P,\cS,\Delta_P';u^{-1},v^{-1},w^{-1}).
%\]

\end{enumerate}
\end{theorem}

In particular, we have the following diagram of invariants
\[\xymatrix{
h^*(P,\cS;u,v,w)   \ar[r]^>>>>>{\substack{u \mapsto uw^{-1} \\ v \mapsto 1}}   \ar[d]^{w \mapsto 1}  & h^*(P;u,w) \ar[d]^{w \mapsto 1}  &\\
 h^*(P,\cS;u,v) \ar[r]^{v \mapsto 1} &   h^*(P;u) \ar[r]^{u \mapsto 1}  & (\dim P)!\vol(P),
}\] 
where $\vol(P)$ is the Euclidean volume of $P$.

Let $\Delta_P$ denote the normal fan to $P$ with all maximal cones removed. The cones $\gamma_Q$ in $\Delta_P$  are in inclusion-reserving correspondence with the positive dimensional faces $Q$ of $P$.
Let $\Delta_P'$ denote a simplicial fan refinement of $\Delta_P$ which exists by the resolution of singularities algorithm for toric varieties \cite[Sec. 2.6]{FulIntroduction}. That is, every cone $\gamma'$ in $\Delta_P'$ is generated by precisely $\dim \gamma'$ rays, and is contained in a cone of $\Delta_P$. We let $\sigma(\gamma')$ denote the smallest cone in $\Delta_P$ containing $\gamma'$, and set 
\[
\Phi(P,\cS,\Delta_P';u,v,w) = %\sum_{\emptyset \ne Q \subseteq P} 
\sum_{  \substack{ Q \subseteq P  \\ \dim Q > 0   } } (-1)^{\dim Q} h^*(Q,\cS|_{Q};u,v,w)  \sum_{ \substack{ \gamma' \in \Delta_P' \\  \sigma(\gamma') = \gamma_Q } }   (uvw^2 - 1)^{\dim \gamma_Q - \dim \gamma'},
\]
and 
\[
\Lambda(P,\cS,\Delta_P';u,v,w) = \sum_{ \gamma' \in \Delta_P'  } (uvw^2 - 1)^{\dim P - \dim \gamma'}   - \Phi(P,\cS,\Delta_P';u,v,w). 
\]

We have the following characterization of the refined limit mixed $h^*$-polynomial is proved in  \cite[Corollary~9.7]{otherpaper}. 

\begin{corollary}\label{c:superimportant}
The refined  limit mixed $h^*$-polynomial as an invariant of polyhedral subdivisions of lattice polytopes is uniquely characterized by the following properties:
\begin{enumerate}
\item\label{m:1} The degree of $h^*(P,\cS;u,v,w)$ as a polynomial in $w$ is at most $\dim P + 1$.  

\item\label{m:2} The refined  limit mixed $h^*$-polynomial specializes to the limit mixed $h^*$-polynomial i.e.
\[
h^*(P,\cS;u,v,1) = h^*(P,\cS;u,v). 
\]

\item\label{m:3} If $\Delta_P'$ denotes a simplicial fan refinement of $\Delta_P$ then for $\Lambda$ defined in terms of the refined limit mixed $h^*$-polynomial as above, we have 
\[
\Lambda(P,\cS,\Delta_P';u,v,w) = (uvw^2)^{\dim P + 1}\Lambda(P,\cS,\Delta_P';u^{-1},v^{-1},w^{-1}).
\]
\end{enumerate}
\end{corollary}

Similarly, the following characterization of the mixed $h^*$-polynomial is given in  \cite[Corollary~9.8]{otherpaper}. 
With the notation above, we set $\Lambda(P,\cS,\Delta_P';u,w) := \Lambda(P,\cS,\Delta_P';uw^{-1},1,w)$. 
Using \eqref{prop2} in Theorem~\ref{t:combinatorics}, we may write this as:
\[
 \sum_{ \gamma' \in \Delta_P'  } (uw - 1)^{\dim P - \dim \gamma'}  - \sum_{  \substack{ Q \subseteq P  \\ \dim Q > 0   } } (-1)^{\dim Q} h^*(Q,\cS|_{Q};u,w)  \sum_{ \substack{ \gamma' \in \Delta_P' \\  \sigma(\gamma') = \gamma_Q } }   (uw - 1)^{\dim \gamma_Q - \dim \gamma'}.
 \] 

\begin{corollary}\label{c:notasimportant}
The mixed $h^*$-polynomial as an invariant of lattice polytopes is uniquely characterized by the following properties:
\begin{enumerate}
\item\label{m:1'} All terms in $h^*(P;u,w)$ have combined degree in $u$ and $w$ at most $\dim P + 1$.  

\item\label{m:2'} The mixed $h^*$-polynomial specializes to the $h^*$-polynomial i.e.
\[
h^*(P;u,1) = h^*(P;u). 
\]

\item\label{m:3'} If $\Delta_P'$ denotes a simplicial fan refinement of $\Delta_P$ then for $\Lambda$ defined in terms of the mixed $h^*$-polynomial as above, we have 
\[
\Lambda(P,\cS,\Delta_P';u,w) = (uw)^{\dim P + 1}\Lambda(P,\cS,\Delta_P';u^{-1},w^{-1}).
\]
\end{enumerate}
\end{corollary}

The following example is computed in \cite[Example~9.10]{otherpaper}:

\begin{example}\label{e:smallterms}
If we write 
\[
h^*(P,\cS,u,v,w) = 1 + uvw^2 \cdot \sum_{0 \le p,q,r \le \dim P  - 1} h^*_{p,q,r}(P,\cS) u^{p}v^{q}w^r,
\]
then we have an explicit description of some of the coefficients of  $h^*(P,\cS,u,v,w)$. 
If $F$ is a cell of $\cS$, then let $\sigma(F)$ denote the smallest face of $P$ containing $F$.
Then for $q,r > 0$,
\[
h^*_{0,q,r}(P,\cS) =   \sum_{\substack{F \in \cS \\\ \dim F = q + 1 \\\ \dim \sigma(F) = r + 1}} \# (\Int(F) \cap M),
\]
\[
h^*_{0,0,r}(P,\cS) =   \sum_{\substack{F \in \cS \\\ \dim F \le 1 \\\ \dim \sigma(F) = r + 1}} \# (\Int(F) \cap M),
\]
\[
\dim P + 1 + h^*_{0,0,0}(P,\cS) =  \sum_{\substack{Q \subseteq P \\\ \dim Q \le 1}} \# (\Int(Q) \cap M). %\sum_{\substack{F \in \cS \\\ \dim F \le 1 \\\ \dim \sigma(F) \le 1}} \# (\Int(F) \cap M).
\]
%For example, when $\dim P = 1$, 
%\[
%h^*(P,\cS,u,v,w) = 1 + uvw^2 \cdot  \# (\Int(F) \cap M).
%\]
%When $\dim P = 2$, 
%\[
%h^*_{0,0,0}(P,\cS) =   \# (\partial P \cap M) - 3,
%\]
%\[
%h^*_{0,0,1}(P,\cS) =  h^*_{1,1,1}(P,\cS) =  \sum_{\substack{F \in \cS, F \nsubseteq \partial P \\\ \dim F \le 1 }} \# (\Int(F) \cap M),
%\]
%\[
%h^*_{0,1,1}(P,\cS) =  h^*_{1,0,1}(P,\cS) =  \sum_{\substack{F \in \cS \\\ \dim F = 2 }} \# (\Int(F) \cap M),
%\]
%\[
%h^*(P,\cS,u,v,w) = 1 + uvw^2 %\cdot
% \big[ h^*_{0,0,0}(P,\cS) + h^*_{0,0,1}(P,\cS)(1 + uv)w + h^*_{0,1,1}(P,\cS)(u +v)w ] .
%\]

Using Property \eqref{prop'} of Theorem~\ref{t:combinatorics},  when $\dim P = 2$, this gives an explicit description of $h^*(P,\cS,u,v,w)$:
\[
h^*(P,\cS,u,v,w) = 1 + uvw^2  \big[ h^*_{0,0,0}(P,\cS) + w\big[(1 + uv)h^*_{0,0,1}(P,\cS) + (u + v)h^*_{0,1,1}(P,\cS)\big].
\]
When $\dim P = 3$, we have 
\[
h^*(P,\cS,u,v,w) = 1 + uvw^2  \big[ h^*_{0,0,0}(P,\cS) + w\big[(1 + uv)h^*_{0,0,1}(P,\cS) + (u + v)h^*_{0,1,1}(P,\cS)\big] 
\]
\[
+ w^2\big[ (1 + (uv)^2)h^*_{0,0,2}(P,\cS) + (u + v)(1 + uv)h^*_{0,1,2}(P,\cS) + (u^2 + v^2)h^*_{0,2,2}(P,\cS) + uvh^*_{1,1,2}(P,\cS)  \big]   \big],  
\]
where each term has an explicit description above except $h^*_{1,1,2}(P,\cS)$. By  \eqref{prop3} of Theorem~\ref{t:combinatorics}, 
$ h^*(P,\cS,1,1,1) = h^*(P,\cS,1) = 6\vol(P)$, and this determines $h^*_{1,1,2}(P,\cS)$ and hence $h^*(P,\cS,u,v,w)$.
\end{example}

%\begin{remark}\label{r:localsymmetry}
%It follows from \eqref{prop'} and \eqref{prop4} in Theorem~\ref{t:combinatorics} that the local limit mixed $h^*$-polynomial satisfies the symmetry $\lc(P, \cS; u,v) = (uv)^{\dim P + 1} \lc(P, \cS; u^{-1},v^{-1})$. 
%\end{remark}

\section{Refined limit mixed Hodge numbers of hypersurfaces} \label{s:hypersurfaces}

The goal of this section is to present a  proof of Theorem~\ref{t:mainhyper} giving a combinatorial formula for the refined limit Hodge-Deligne polynomial of a sch\"{o}n hypersurface in $(\K^*)^n$ which is interpreted as a family of hypersurfaces.
We first reprove a theorem of Danilov-Khovanski{\u\i} for the $\chi_y$-characteristic of a complex hypersurface in terms of the $h^*$-polynomial of its Newton polytope.  Then we give combinatorial formulas of the following progressively finer cohomological invariants: the Hodge-Deligne polynomial of a generic fiber; the limit Hodge-Deligne polynomial, the limit Hodge-Deligne polynomial of a smooth compactification of the family of hypersurfaces, and then the refined limit Hodge-Deligne polynomial.  We will make use of the fact that the cohomology of a hypersurface is tightly constrained by Poincar\'{e} duality and the weak Lefschetz theorem.

\subsection{Tropical geometry for hypersurfaces}\label{ss:tropical}

Let $X^\circ =  \{ \sum_{u \in M} \alpha_u x^u = 0 \} \subset T \cong (\K^*)^n$ be a sch\"on hypersurface. 
The \define{Newton polytope} $P$ of $X^\circ$ is the convex hull of $\{ u \in M \mid \alpha_u \ne 0 \}$. 
 Note that  $P$ may be viewed as
 a full-dimensional lattice polytope in the translation $M$ of the saturation of its integer affine span in $\Z^n$ to the origin, and  $X^\circ \cong X' \times (\K^*)^k$,  
 for some $k$, where $X' \subseteq \Spec \K[M] $ is a sch\"on hypersurface with Newton polytope $P$. 
 Hence we may and will assume that $\dim P = n$. % and $\dim X^\circ = d = n -1$.

Tropical geometry of hypersurfaces reduces to the study of Newton polytopes and polyhedral subdivisions \cite{GKZ,RSTFirst}. Recall that the field $\K$ has a natural valuation by considering the vanishing order of a function on $\D^*$ at the origin.  
With the notation above, the function 
$P \cap \Z^n \rightarrow \Z$, $u \mapsto \ord(\alpha_u)$ induces a regular, lattice subdivision $\cS$ of $P$.  Explicitly, the cells of $\cS$ are the projections 
of the bounded faces of the convex hull of $\operatorname{UH}=\{ (u, \lambda) \mid \alpha_u \ne 0, \lambda \ge \ord(\alpha_u) \}$ in $\R^n \times \R$,
and the bounded faces of $\operatorname{UH}$ are the graph of a function $\omega:P\rightarrow\R$.  Restricting to $P\cap \Z^n$, we get a height function.  There is a dual complex associated to the height function that generalizes the normal fan.  The cells of this complex are in inclusion-reversing bijective correspondence with the cells of $\cS$.   See \cite[9.11]{Gubler} for details.

\begin{remark}\label{r:existence}
Since the initial degeneration $\init_wX^{\circ}$ of a hypersurface is given by the corresponding initial form of its defining polynomial,  for a generic choice of coefficients (in a certain analytic topology), 
a hypersurface with a given height function is sch\"on, i.e. all initial degenerations are smooth.  
%by Proposition \ref{prop:schon}, for a generic choice of coefficients (in a certain analytic topology), a hypersurface with a given height function is sch\"on.  
Hence every pair $(P,\cS)$, where $\cS$ is a regular, lattice polyhedral subdivision of a lattice polytope $P$ arises from the construction above for some sch\"on hypersurface. 
See \cite[Section~8.1]{HelmKatz} for a more detailed discussion of genericity and sch\"{o}nness. 
\end{remark}

The tropicalization $\Trop(X^\circ)$ is supported on the non-maximal-dimensional skeleton of the dual complex \cite[Section~3]{RSTFirst} to $\cS$.  The restriction of the dual complex to $\Trop(X^\circ)$ gives a polyhedral structure $\Sigma$.
%The corresponding recession fan of $\Trop(X^\circ)$ %may and will be chosen to be 
%is the fan obtained from the normal fan of $P$ by removing its maximal cones. 
With the notation of Section~\ref{s:combinatorial}, 
the recession fan $\Delta = \Delta_P$ of $\Sigma$ is %equal to the normal fan $\Delta_P$ of $P$ with the maximal cones removed. 
the normal fan of $P$ with the maximal cones removed.
Recall from Section~\ref{s:proof} that we may define a toric scheme $\P(\Sigma)_\O$ over $\O$ from $\Sigma$ with generic fiber equal to the toric variety $\P(\Delta)_\K$. 
Let $\X$ denote the closure of $X^\circ$ in $\P(\Sigma)_\O$, and let $X_{\Delta}$ and $X_0$ denote the generic fiber and central fiber of $\X$ respectively.  
Then we can write the stratifications of $X_{\Delta}$ and $X_0$ in dual language with respect to the Newton polytope and subdivision as the following:
\[
X_{\Delta} = \bigcup_{ \substack{ Q \subseteq P \\ \dim Q > 0  } } X^\circ_{Q}, \: X_0 = \bigcup_{  \substack{ F \in \cS \\ \dim F > 0  }    } X^\circ_{F}
\]
The fixed non-zero fiber $X^\circ_{\gen}$ is  a sch\"on hypersurface with Newton polytope $P$ in its corresponding complex torus, which we denote as $T_{\gen}$. For every cell $F$ of $\cS$ with $\dim F > 0$, the 
corresponding complex variety $X^\circ_{F}$ is a  complex sch\"on hypersurface with Newton polytope $F$, and, %respect to a complex torus, and 
if $w$ lies in the relative interior of the cell in $\Sigma$ corresponding to $F$, then 
\[
\init_w X^\circ\cong  X_F^\circ  \times (\C^*)^{\dim P - \dim F}. %(\X\cap U_F).
\]
When $\dim F = 0$, $\init_w X^\circ = X_F^\circ = \emptyset$, and the corresponding motivic invariants are zero.
%We conclude that the formula for the limit Hodge-Deligne polynomial in Corollary~\ref{c:main} translates into the following result. 
We conclude that Theorem~\ref{t:comp} translates into the following corollary. 

\begin{corollary} \label{c:hyper} 
Let $X^\circ \subseteq (\K^*)^n$ be a sch\"{o}n hypersurface, with associated Newton polytope and polyhedral subdivision $(P,\cS)$ and $\dim P = n$. 
Then the  motivic nearby fiber of $X^\circ$ is given by
%\[
%E(X^\circ_\infty;u,v) = \sum_{\substack{F \in \cS \\\ F \nsubseteq \partial P}} E(X_{F}^\circ;u,v)(1 - uv)^{\dim P - \dim F},
%\]
\[
\psi_{X^\circ} =  \sum_{\substack{F \in \cS \\\ F \nsubseteq \partial P}}  [X_{F}^\circ](1 - \L)^{\dim P - \dim F},
\] 
where $\partial P$ denotes the boundary of $P$, $\L := [\A^1] \in K_0(\Var_\C)$, and $X_F^\circ$ is a complex sch\"on hypersurface with Newton polytope $F$.
\end{corollary}

\subsection{%A new proof of t
The $\chi_y$-characteristic of a complex %sch\"on 
hypersurface}\label{s:chiy}

We apply Corollary~\ref{c:hyper} %our theory of the tropical motivic nearby fiber 
to %derive
give a new proof of
 a formula of Danilov-Khovanski{\u\i} \cite[Section~4]{DKAlgorithm} for the $\chi_y$-characteristic of sch\"{o}n hypersurfaces in $(\C^*)^n$.  
%Danilov-Khovanski{\u\i} use their formula in connection with the Hard Lefschetz theorem and Poincar\'{e} duality to compute all the Hodge numbers.  
%Our method is different from theirs in that they use an adjunction exact sequence on the ambient toric varieties while we attempt to degenerate the hypersurface into a union of hyperplanes.

%\eric{changed $X$'s to $V$'s below}

\begin{remark}\label{r:independence}
The fact that the Hodge-Deligne polynomial of a sch\"on hypersurface of a complex torus is determined by its Newton polytope can be seen directly. One considers 
the closure $V$ of a sch\"on hypersurface $V^\circ$ given by a Laurent polynomial with Newton polytope $P$  in a toric resolution of the complex toric variety determined by $P$.  It is a smooth variety. 
Since Hodge numbers are locally constant through families of smooth varieties, the Hodge-Deligne polynomial of $V$ is independent of the choice of polynomial.  
The result can then be deduced from the motivic nature of the Hodge-Deligne polynomial.

\end{remark}

Let $V^\circ$ be a sch\"on hypersurface of a complex torus given by a polynomial with Newton polytope $P$.  We may suppose $\dim P = \dim V^\circ + 1$. Recall from Remark~\ref{r:existence} and Remark~\ref{r:independence}
that the Hodge-Deligne polynomial of $V^\circ$ only depends on $P$, and that given any $P$, there exists a corresponding sch\"on hypersurface $V^\circ$. Hence we may define
\[ E(V_P^\circ;u,v) := E(V^\circ;u,v).\]
If $P$ is empty, then we let $V_P^\circ$ be the empty set. 
%Let us determine the value of $E(V(P)^\circ;u,v)$ on unimodular simplices.
To identify the $\chi_y$-characteristic, we build a valuation out of it (see Definition~\ref{d:valuation}).

\begin{lemma}\label{l:valuation}
The map 
\begin{eqnarray*}
\cp_{\Z^n}&\rightarrow&\Z[[u]]\\
P&\mapsto &\frac{E(V_P^\circ;u,1)}{(u-1)^{\dim P+1}}
\end{eqnarray*}
is a valuation on the set $\cp_{\Z^n}$ of lattice polytopes in $\Z^n$.
\end{lemma}

\begin{proof}
Properties (\ref{v:c2}) and (\ref{v:c3}) in Definition~\ref{d:valuation} are clearly satisfied so we must show Property~(\ref{v:c1}).  Let $\cS$ be a regular lattice polyhedral subdivision of $P$. %induced by a height function. 
By Remark~\ref{r:existence}, there exists a  sch\"on hypersurface $X^\circ \subset %T = (\K^*)^n$ 
 (\K^*)^{\dim P}$ with corresponding Newton polytope $P$ and polyhedral subdivision $\cS$.   
This hypersurface satisfies $E(X^\circ_{\gen};u,1)=E(V_P^\circ;u,1)$.
By Corollary~\ref{c:hyper} and \eqref{e:xy}, we obtain %:
\[
E(V_P^\circ;u,1)
=  \sum_{\substack{F \in \cS \\\ F \nsubseteq \partial P}}   E(V_F^\circ;u,1)(1 - u)^{\dim P - \dim F}
%=\sum_{\substack{F \in \cS \\\ F \nsubseteq \partial P}}   E(V_F^\circ;u,1)(1 - u)^{\dim P - \dim F}. 
\]
If we divide by $(u-1)^{\dim P+1}$, we get
\[
\frac{E(V_P^\circ;u,1)}{(u-1)^{\dim P+1}} =
 \sum_{\substack{F \in \cS \\\ F \nsubseteq \partial P}}  (-1)^{\dim P - \dim F}      \frac{E(V_F^\circ;u,1)}{(u - 1)^{\dim F+1}}.
\]
%\eric{changed}
\end{proof}

With the notation of Section~\ref{s:combinatorial}, we obtain a new proof of Danilov and Khovanski{\u\i}'s theorem. 

\begin{theorem} \label{t:xy} \cite[Sec. 4]{DKAlgorithm} 
Let  $P$ be a non-empty lattice polytope and let $V_P^\circ$ be a complex sch\"on hypersurface with Newton polytope $P$. Then 
%If $P\neq\emptyset$, then 
%we have the following formula for the $\chi_y$-characteristic of $V(P)^\circ$:
%\[uE(V(P)^\circ;u,1)=(u-1)^{\dim P}+(-1)^{\dim P+1} h_P^*(u),\]
we have the following formula for the $\chi_y$-characteristic of $V_P^\circ$: 
\[uE(V_P^\circ;u,1)=(u-1)^{\dim P}+(-1)^{\dim P+1} h^*(P;u),\]
where $h^*(P;u)$ is the $h^*$-polynomial of $P$. 
\end{theorem}

\begin{proof}
We continue with the notation of Lemma~\ref{l:valuation}.
By dividing both sides of the equation by $(u-1)^{\dim P+1}$, it suffices to establish the following:
\[
\frac{uE(V_P^\circ;u,1)}{(u-1)^{\dim P+1}}=\frac{f_0(P)}{u-1}+  \sum_{m \ge 0} f_P(m) u^m %Ehr_P(u).
\]
By Lemma~\ref{l:valuation} and Example~\ref{e:aval}, both sides are valuations. By Lemma~\ref{l:simplex}, we need only check the case of  unimodular simplices $\Delta_l$. In that case,  a straightforward computation \cite[Sec. 2.3]{BRComputing} gives $h^*(\Delta_l) = 1$ and we need to check that 
\[
uE(V_{\Delta_l}^\circ;u,1) = (u-1)^{l}+(-1)^{l+1}.
\]
We prove this by induction.  For $l=0$, both sides of the equation are $0$. For $l\geq 1$, $V_{\Delta_l}^\circ$  is the intersection of a generic hyperplane in $\P^l$ with $(\C^*)^l$.  This is isomorphic to the complement of $l+1$ generic hyperplanes in $\P^{l-1}$.  By treating $l$ of these hyperplanes as coordinate hyperplanes and the last one as some generic hyperplane, we get the  motivic relation $[V_{\Delta_l}]=[(\C^*)^{l-1}]-[V_{\Delta_{l-1}}]$.  Because $V^\circ\mapsto E(V^\circ;u,1)$ is motivic and $E((\C^*)^{l-1};u,1)=(u-1)^{l-1}$, we have
\begin{eqnarray*}
uE(V_{\Delta_l}^\circ;u,1)&=&u(u-1)^{l-1}-(u-1)^{l-1}-(-1)^{l}\\
&=& (u-1)^{l}+(-1)^{l+1}.
\end{eqnarray*}

%We claim that 
%\[\frac{uE(V(\Delta_l)^\circ;u,1)}{(u-1)^{l+1}}=\frac{1}{u-1}+\frac{1}{(1-u)^{l+1}}.\]
%We prove this by induction.  For $l=0$, both sides of the equation are $0$.
%For $l\geq 1$, $V(\Delta_l)$ is the intersection of a generic hyperplane in $\P^l$ with $(\C^*)^l$.  This is isomorphic to the complement of $l+1$ generic hyperplanes in $\P^{l-1}$.  By treating $l$ of these hyperplanes as coordinate hyperplanes and the last one as some generic hyperplane, we get the  motivic relation $[V(\Delta_l)]=[(\C^*)^{l-1}]-[V(\Delta_{l-1})]$.  Because $V\mapsto E(V;u,1)$ is motivic and $E((\C^*)^{l-1};u,1)=(u-1)^{l-1}$, we have
%\begin{eqnarray*}
%E(V(\Delta_l)^\circ;u,1)&=&(u-1)^{l-1}-\frac{(u-1)^{l-1}+(-1)^{l}}{u}\\
%&=&\frac{(u-1)^{l}+(-1)^{l+1}}{u}
%\end{eqnarray*}

%Now, we claim that 
%\[\frac{L_0(\Delta_l)}{u-1}+\Ehr_{\Delta_l}(u)=\frac{1}{u-1}+\frac{1}{(1-u)^{l+1}}.\]
%  For $l=0$, this is trivial.  For $l\geq 1$, it follows from a straightforward computation \cite[Sec. 2.3]{BRComputing} that
%$\Ehr_{\Delta_l}(u)=\frac{1}{(1-u)^{l+1}},\ h^*_{\Delta_l}=1.$
\end{proof}

By specializing the above theorem to $u=1$ and using the fact that
$h^*(P;1) = (\dim P)!\vol(P)$ where $\vol(P)$ is the Euclidean volume of $P$, we get the following well-known result of Kouchnirenko \cite{Kouchnirenko}:

\begin{corollary}
Let  $P$ be a non-empty lattice polytope and let $V_P^\circ$ be a sch\"on hypersurface with Newton polytope $P$. Then 
we have the following formula for the topological Euler characteristic of $V_P^\circ$: \[ e(V_P^\circ)=(-1)^{\dim P+1} (\dim P)! \vol(P) .\]
\end{corollary}

\subsection{A Danilov-Khovanski{\u\i} type algorithm}\label{ss:DK}

In \cite{DKAlgorithm}, Danilov and Khovanski{\u\i} use their formula for the $\chi_y$-characteristic in Theorem~\ref{t:xy} 
in connection with the weak Lefschetz theorem and Poincar\'{e} duality to give an algorithm to compute the Hodge-Deligne polynomial of a complex sch\"on hypersurface. 
We use an analogous approach to provide an algorithm to compute the refined limit Hodge-Deligne polynomial of a sch\"on hypersurface from the limit Hodge-Deligne polynomial. 
%We next provide a characterization of the refined limit Hodge-Deligne polynomial that is analogous to Danilov and Khovanski{\u\i}'s characterization of the 
%Hodge-Deligne polynomial of a complex sch\"on hypersurface in \cite{DKAlgorithm}. 
We continue with the notation from earlier in this section. 

We consider the cohomology with compact supports of the complex variety $X^{\circ}_{\gen} \subseteq T_{\gen}$, and set $n = \dim T_{\gen}$. The following weak Lefschetz result implies that
the only interesting cohomology is in middle dimension. 

\begin{proposition}\cite[Proposition 3.9]{DKAlgorithm}\label{p:Gysin}
The Gysin map $H^k_c (X^{\circ}_{\gen}) \rightarrow H^{k + 2}_c (T_{\gen})$ is an isomorphism for $k > n - 1$, and a surjection for  $k = n - 1$. 
Since $X^{\circ}_{\gen}$ is affine, $H^k_c (X_{\gen}^\circ) = 0$ for $k < n - 1$. 
\end{proposition}

Indeed, the Gysin map above is a morphism of mixed Hodge 
structures of type $(1,1)$, and hence the (usual) mixed Hodge structure on $H^k_c (X^{\circ}_{\gen})$ is known for $k \ne n - 1$ by Example~\ref{e:torus}. 
Following \cite{BatVariations}, we define the  \define{primitive cohomology} of  $X^\circ_{\gen}$ to be
\[
H^{n - 1}_{c,\prim} X^\circ_{\gen} := \ker[H^{n - 1}_c X^\circ_{\gen}  \rightarrow H^{n + 1}_c T_{\gen}],
\]
with the induced mixed Hodge structure.  Since the Gysin map varies naturally in families over $\D^*$, it commutes with the monodromy operator, and so
by Example~\ref{e:torusfiber}, the
corresponding nilpotent operator $N$ preserves the primitive cohomology of $X^\circ_{\gen}$.

%\begin{lemma} Let $M_\bullet$ be the monodromy filtration on $H^{d}_c X^\circ_{\gen}$.  Then $M_\bullet  \cap H^{n - 1}_{c,\prim} X^\circ_{\gen}$ is the monodromy filtration on $H^{n - 1}_{c,\prim} X^\circ_{\gen}$.
%\end{lemma}

%\begin{proof}
%Because $T$  is the constant family over the punctured disc, $N=0$ on $H^*_c(T_{\gen})$.   The filtration $M_\bullet  \cap H^{d}_{c,\prim} X^\circ_{\gen}$ coincides with the monodromy filtration explicitly constructed in \cite[Proposition~(1.6.1)]{WeilII}.

%\end{proof}

It follows that the refined limit Hodge-Deligne polynomial $E(X_\infty^\circ;u,v,w)$ determines and is determined by the refined limit Hodge numbers of the primitive cohomology of $X_\infty^\circ$. 
In particular, we have the following lemma:

\begin{lemma}\label{l:weaklef}
Let $X^\circ \subseteq (\K^*)^n$ be a sch\"{o}n hypersurface, with associated Newton polytope and polyhedral subdivision $(P,\cS)$. 
Then, as a polynomial in $w$,  $uvw^2E(X_\infty^\circ;u,v,w)$ has the same coefficient as  
$(uvw^2 - 1)^{\dim P + 1}$ in all degrees strictly greater than  $\dim P + 1$. 
\end{lemma}
\begin{proof}
Since $X^\circ_{\gen}$ is a smooth complex variety, the graded pieces of the Deligne weight filtration $\Gr_r^W H_c^m (X^\circ_{\gen})$ are zero for $r > m$ by e.g. \cite[Thm 5.39]{PSMixed}. In particular, the contributions
from the primitive cohomology of $X^\circ_{\gen}$ to $E(X_\infty^\circ;u,v,w)$ all have degree at most $\dim P - 1$ in $w$.  The result then follows from the above discussion and Example~\ref{e:torusfiber}. 
\end{proof}

The above lemma may be viewed as a generalization of the corresponding statement for the Hodge-Deligne polynomial, due to Danilov and Khovanski{\u\i}, which follows by the exact same argument as above. 

\begin{lemma}\label{l:weaklef2}\cite[Sec. 3.11]{DKAlgorithm}
Let $X^\circ \subseteq (\K^*)^n$ be a sch\"{o}n hypersurface, with associated Newton polytope $P$.  Then the coefficient of $u^pw^q$ in $uwE(X_{\gen}^\circ;u,w)$ equals the coefficient of $u^pw^q$ in $(uw - 1)^{\dim P + 1}$ for $p + q > \dim P + 1$. 
%, as a polynomial in $w$,  $uwE(X_{\gen}^\circ,u,w)$ has the same coefficient as  
%$(uw - 1)^{\dim P + 1}$ in all degrees greater than  $\dim P$. 
\end{lemma}

We next explain the use of Poincar\'{e} duality. Recall that the recession fan 
$\Delta_P$ is the normal fan to $P$ with all maximal cones removed, with cones $\gamma_Q$  in inclusion-reserving correspondence with the positive dimensional faces $Q$ of $P$.  As in Section~\ref{s:combinatorial}, let $\Delta_P'$ denote a simplicial fan refinement of $\Delta_P$, and let $\sigma(\gamma')$ denote the smallest cone in $\Delta_P$ containing a cone $\gamma'$ in $\Delta_P'$. 
 Then we have an induced proper, birational map of toric varieties over $\K$, $\pi: \P(\Delta_P')_\K \rightarrow \P(\Delta_P)_\K$, which, by standard toric geometry, is locally a projection in the sense that if 
 $\P(\Delta_P')_\K = \bigcup_{\gamma' \in \Delta_P'} U_{\gamma'}$ and $\P(\Delta_P)_\K = \bigcup_{ Q \subseteq P , \dim Q > 0  } U_{\gamma_{Q}}$ are unions of the toric varieties into torus orbits, then 
 $\pi|_{U_{\gamma'}}$ is given by
\[
\pi|_{U_{\gamma'}}: U_{\gamma'} \cong  U_{\sigma(\gamma')}  \times (\K^*)^{\dim \sigma(\gamma') - \dim \gamma'} \rightarrow U_{\sigma(\gamma')}.
\]
Let $X_P'$ and $X_P$ denote the closure of $X^\circ$ in the toric varieties  $\P(\Delta_P')_\K$ and $\P(\Delta_P)_\K$ respectively.  Then $X_{P}'$ is proper and has at worst orbifold singularities.
The possibly singular variety $X_P$ has a stratification into sch\"on subvarieties 
\[
X_P = \bigcup_{ \substack{ Q \subseteq P  \\ \dim Q > 0   } } X_Q^\circ,
\]
where $X^\circ = X_P^\circ$, and $X_Q^\circ$ corresponds to the pair $(Q, \cS|_Q)$.  We conclude that 
\begin{equation}\label{e:simplicialdecomp}
E(X_{P,\infty}'; u,v,w) = %\sum_{\emptyset \ne Q \subseteq P} 
\sum_{  \substack{ Q \subseteq P  \\ \dim Q > 0   } }E(X_{Q,\infty}^{\circ}; u,v,w) \sum_{ \substack{ \gamma' \in \Delta_P' \\  \sigma(\gamma') = \gamma_{Q }} }  (uvw^2 - 1)^{\dim \gamma_Q - \dim \gamma'}.
\end{equation}
%where $E(X_{Q,\infty}^{\circ}; u,v,w)  = 0$ if $\dim Q = 0$. 
Since $X_{P}'$ is proper and has at worst orbifold singularities, Poincar\'{e} duality \cite[Prop 6.19]{PSMixed} implies that
\begin{equation}\label{e:Poincare}
E(X_{P,\infty}'; u,v,w) = (uvw^2)^{\dim P - 1}E(X_{P,\infty}'; u^{-1},v^{-1},w^{-1}). 
\end{equation}

We conclude that we have the following algorithm to determine $E(X_{\infty}^\circ; u,v,w)$ from $E(X_{\infty}^\circ; u,v,1) = E(X_{\infty}^\circ; u,v)$, using induction on dimension. Consider  $E(X_{\infty}^\circ; u,v,w)$ as a polynomial in $w$. 
Firstly, Lemma~\ref{l:weaklef} implies that we know $E(X_{\infty}^\circ; u,v,w)$ in all degrees strictly greater than $\dim P - 1$. Secondly, by induction on dimension and \eqref{e:simplicialdecomp}, we know 
$E(X_{P,\infty}'; u,v,w)$ in all degrees strictly greater than $\dim P - 1$, and by \eqref{e:Poincare}, we know $E(X_{P,\infty}'; u,v,w)$ and hence $E(X_{\infty}^\circ; u,v,w)$ in all degrees strictly less than $\dim P - 1$. Finally, 
 $E(X_{\infty}^\circ; u,v,1)$ now determines $E(X_{\infty}^\circ; u,v,w)$ in degree $\dim P - 1$. 

\begin{remark}\label{r:mixedcase}
The same argument gives the Danilov and Khovanski{\u\i} algorithm to determine the Hodge-Deligne polynomial $E(X_{\infty}^\circ; uw^{-1},1,w) = E(X_{\gen}^\circ; u,w)$  from the $\chi_y$-characteristic 
 $E(X_{\infty}^\circ; u,1,1) = E(X_{\gen}^\circ; u,1)$. Explicitly, Lemma~\ref{l:weaklef2} implies that we know the coefficient of $u^pw^q$ in $E(X_{\gen}^\circ;u,w)$ for $p + q > \dim P - 1$.  
 %$E(X_{\gen}^\circ; u,w)$ in all degrees greater than $\dim P - 1$. 
 Secondly, by induction on dimension and \eqref{e:simplicialdecomp}, we know  the coefficient of $u^pw^q$ in 
$E(X_{P,\gen}'; u,w)$ for $p + q > \dim P$, and 
%in all degrees greater than $\dim P - 1$, and 
by \eqref{e:Poincare}, we know the coefficient of $u^pw^q$ in $E(X_{P,\gen}'; u,w)$ and hence $E(X_{\gen}^\circ; u,w)$ in all degrees strictly less than $\dim P - 1$. Finally, 
 $E(X_{\gen}^\circ; u,1)$ now determines  the coefficient of $u^pw^q$ in $E(X_{\gen}^\circ; u,w)$ when $p + q = \dim P - 1$. %in degree $\dim P - 1$. 
\end{remark}

%%%%%%%%%%%%%%%%%%%%%%

\subsection{A formula for the refined limit Hodge-Deligne polynomial} \label{ss:refinedhyper}

We now complete the proof of Theorem~\ref{t:mainhyper}, and deduce an explicit description of the refined limit mixed Hodge numbers of a sch\"on hypersurface. We will see that the proof reduces
to some combinatorial results which are proved in \cite{otherpaper}. We also state some immediate consequences of the theorem. 

Let $X^\circ \subseteq (\K^*)^n$ be a sch\"{o}n hypersurface, with associated Newton polytope and polyhedral subdivision $(P,\cS)$ and $\dim P = n$. 
We will work our way through the diagram
\[\xymatrix{
E(X_\infty^\circ; u,v,w)  \ar[r] \ar[d] & E(X_{\gen}^\circ; u,w) \ar[d] &\\
E(X_\infty^\circ; u,v) \ar[r]& E(X_{\gen}^\circ; u,1)   \ar[r]   & e(X_{\gen}^\circ),
}\] 
In Section~\ref{s:chiy}, we proved the formula 
\[uE(X_{\gen}^\circ;u,1)=(u-1)^{\dim P}+(-1)^{\dim P+1} h^*(P;u),\]
where $h^*(P;u)$ is the $h^*$-polynomial of $P$. We claim that 
\[
uwE(X_{\gen}^\circ;u,w)=(uw-1)^{\dim P}+(-1)^{\dim P+1} h^*(P;u,w). 
\]
 This is the Borisov-Mavlyutov formula for the Hodge-Deligne polynomial \cite{BMString}. We will prove this formula using the method of \cite{StaMirror}. 
 Indeed, we only need to verify that the proposed formula satisfies the algorithm of Remark~\ref{r:mixedcase}. Recall that the algorithm consists of three parts: weak Lefschetz, specialization and Poincar\'{e} duality.
 That the proposed formula satisfies the weak Lefschetz property (Lemma~\ref{l:weaklef2}) follows from \eqref{m:1'} in Corollary~\ref{c:notasimportant}. The fact that the proposed formula specializes to the formula for
 $E(X_{\gen}^\circ;u,1)$ when setting $w = 1$ follows from  \eqref{m:2'} in Corollary~\ref{c:notasimportant}. Finally, that the proposed formula satisfies the Poincar\'{e} duality property follows by 
 substitution into \eqref{e:simplicialdecomp} (after specializing $u \mapsto uw^{-1}, v \mapsto 1$) and \eqref{m:3'} in Corollary~\ref{c:notasimportant}. %, using the fact that $h^*(P,u,w) = 1$ if $\dim P = 0$.
 
 To determine the limit Hodge-Deligne polynomial, we note that Corollary~\ref{c:hyper} specializes under application of the Hodge-Deligne map to the formula
 \[
 E(X_\infty^\circ; u,v)  =   \sum_{\substack{F \in \cS \\\ F \nsubseteq \partial P}}  E(X_{F}^\circ;u,v)(1 - uv)^{\dim P - \dim F}. 
 \]
 Substituting the Borisov-Mavlyutov formula for the Hodge-Deligne polynomial yields, using Lemma~\ref{l:euler}, 
 \begin{align*}
 uvE(X_\infty^\circ; u,v)  &=   \sum_{\substack{F \in \cS \\\ F \nsubseteq \partial P}}  \big[ (uv-1)^{\dim F}+(-1)^{\dim F+1} h_F^*(u,v) \big] (1 - uv)^{\dim P - \dim F} \\
  &=   (uv - 1)^{\dim P} + (-1)^{\dim P + 1} \sum_{\substack{F \in \cS \\\ F \nsubseteq \partial P}} h_F^*(u,v)  (uv - 1)^{\dim P - \dim F}. 
 \end{align*}
Now \eqref{prop6} in Theorem~\ref{t:combinatorics} gives our desired formula
 \[
 uvE(X_\infty^\circ; u,v)  =   (uv - 1)^{\dim P} + (-1)^{\dim P + 1} h^*(P,\cS;u,v). 
 \]
 Finally, we want to prove 
 \[
 uvw^2E(X_\infty^\circ; u,v,w)  =   (uvw^2 - 1)^{\dim P} + (-1)^{\dim P + 1} h^*(P,\cS;u,v,w). 
 \]
 It remains to show that the proposed formula satisfies the three parts of the algorithm in Section~\ref{ss:DK}.  That the proposed formula satisfies the weak Lefschetz property (Lemma~\ref{l:weaklef}) follows from \eqref{m:1} in Corollary~\ref{c:superimportant}. The fact that the proposed formula specializes to the above formula for $E(X_\infty^\circ; u,v)$ when setting $w = 1$ follows from  
 \eqref{m:2} in Corollary~\ref{c:superimportant}.  That the proposed formula satisfies the Poincar\'{e} duality property follows by 
 substitution into \eqref{e:simplicialdecomp} and then applying \eqref{m:3} in Corollary~\ref{c:superimportant}.   %,  using the fact that $h^*(P,\cS,u,v,w) = 1$ if $\dim P = 0$.

 Using our description of the cohomology of $X_{\gen}^\circ$ in Section~\ref{ss:DK} and Section~\ref{ss:tropical} together with the above formula, we immediately deduce the following corollary. 
The second two statements below follow from \eqref{prop1} and \eqref{prop4} in Theorem~\ref{t:combinatorics} respectively. We refer the reader to Example~\ref{e:smallterms} and Theorem~\ref{t:combinatorics} for explicit combinatorial descriptions of the invariants below
in the cases when $n = 2,3$. 
 
 \begin{corollary}\label{c:explicit}
Let $X^\circ \subseteq (\K^*)^n$ be a sch\"{o}n hypersurface, with associated Newton polytope and polyhedral subdivision $(P,\cS)$ and $\dim P = n$. Then the refined limit mixed Hodge numbers
associated to the primitive cohomology of $X^\circ$ are given by
\[
h^*(P,\cS;u,v,w) = 1 + uvw^2 \sum_{p,q,r} h^{p,q,r}(H_{\prim,c}^{\dim P - 1}(X^\circ_\infty)) u^p v^q w^r.
\]
In particular, the corresponding limit mixed Hodge numbers are given by
\[
h^*(P,\cS;u,v) = 1 + uv \sum_{p,q} h^{p,q}(H_{\prim,c}^{\dim P - 1}(X^\circ_\infty)) u^p v^q.
\]
Moreover, the limit mixed Hodge numbers of $\Gr^W_{\dim P - 1} H_{\prim,c}^{\dim P - 1}(X^\circ_\infty)$ are given by 
\[
\lc(P, \cS;u,v) = uv\sum_{p,q} h^{p,q}(\Gr^W_{\dim P - 1} H_{\prim,c}^{\dim P - 1}(X^\circ_\infty)) u^p v^q.  
\]
 \end{corollary}

As in Corollary~\ref{c:partial}, the  motivic nature of the invariants above means that we obtain formulas for invariants of 
partial compactifications of sch\"on hypersurfaces. We now state this explicitly for possible future reference.   

Let $X^\circ \subseteq (\K^*)^n$ be a sch\"{o}n hypersurface, with associated Newton polytope and polyhedral subdivision $(P,\cS)$ and $\dim P = n$.
With the notation of Section~\ref{ss:DK}, the recession fan $\Delta_P$ has cones $\gamma_Q$ in inclusion-reserving correspondence with the positive dimensional faces $Q$ of $P$. Let $\tilde{\Delta}_P'$ denote a fan refinement (not necessarily simplicial) of a subfan $\tilde{\Delta}_P$ of $\Delta_P$. We let $\sigma(\gamma')$ denote the smallest cone in $\tilde{\Delta}_P$ containing $\gamma'$. 
Let $\tilde{X}_P'$ denote the closure of $X^\circ$ in the toric variety  $\P(\tilde{\Delta}_P')_\K$ over $\K$.  
Let  $\tilde{X}_P'$ and $\tilde{X}_P$ denote the closure of $X^\circ$ in the toric varieties  $\P(\tilde{\Delta}_P')_\K$ and $\P(\Delta_P)_\K$ respectively. 
Then $\tilde{X}_P'$ and $\tilde{X}_P$ have toroidal singularities, and $X_P$ has a stratification into sch\"on subvarieties 
\[
X_P = \bigcup_{ \substack{ Q \subseteq P  \\ \dim Q > 0   } } X_Q^\circ,
\]
where $X_Q^\circ$ corresponds to the pair $(Q, \cS|_Q)$.
 The arguments of Section~\ref{ss:DK} imply the following expressions
for the refined limit Hodge-Deligne polynomial and motivic nearby fiber of $\tilde{X}_P'$ respectively. 
\[
E(\tilde{X}_{P,\infty}'; u,v,w) = \sum_{  \substack{ Q \subseteq P  \\ \dim Q > 0   } } E(X_{Q,\infty}^{\circ}; u,v,w) \sum_{ \substack{ \gamma' \in \Delta_P' \\  \sigma(\gamma') = \gamma_Q } }  (uvw^2 - 1)^{\dim \gamma_Q - \dim \gamma'},
\]
\[
\psi_{\tilde{X}_{P}'} = \sum_{  \substack{ Q \subseteq P  \\ \dim Q > 0   } } \psi_{X_{Q}^{\circ} } \sum_{ \substack{ \gamma' \in \Delta_P' \\  \sigma(\gamma') = \gamma_Q } }  (\L - 1)^{\dim \gamma_Q - \dim \gamma'}.
\]

Finally, combinatorial expressions for $E(X_{Q,\infty}^{\circ}; u,v,w) $ and $\psi_{X_{Q}^{\circ} }$ are given in  Theorem~\ref{t:mainhyper} and Corollary~\ref{c:hyper} respectively. Note that we may allow 
the summations above to run over all non-empty faces $Q$ of $P$ since $X_{Q}^{\circ} = \emptyset$ when $\dim Q = 0$.

\begin{remark}\label{r:singular}
In the case when $\Delta_P = \tilde{\Delta}_P' = \tilde{\Delta}_P$ above, $X_P := \tilde{X}_P'$ is the closure of $X^\circ$ in the toric variety  $\P(\Delta_P)_\K$ over $\K$. In this case, $X_P$ is proper but singular in general. 
Interestingly, we have a purely combinatorial expression for $E(X_{P,\infty};u,v,w)$ above, although one can not hope to obtain a combinatorial expression for the Betti numbers of $X_P$, nevermind the refined limit mixed Hodge numbers of $X_P$ since the Betti numbers of the ambient toric variety are not combinatorial \cite{Feichtner,McConnell}.
We obtain the following expression for the motivic nearby fiber
\[
\psi_{X_P} = \sum_{ \emptyset \ne Q \subseteq P } \psi_{X_{Q}^{\circ} } = \sum_{\emptyset \ne F \in \cS}  [X_{F}^\circ](1 - \L)^{\dim \sigma(F) - \dim F},
\]
where $\sigma(F)$ denotes the smallest face of $P$ containing $F$. 
If we further assume that $\Delta_P$ is smooth, then the above expression for the motivic nearby fiber appeared in \cite[Section~6]{KatzStapledon}.

\end{remark}

Finally, we present the following application of Theorem~\ref{t:mainhyper}.

\begin{example}\label{e:stringy}
Let $X^\circ \subseteq (\K^*)^n$ be a sch\"{o}n hypersurface, with associated Newton polytope and polyhedral subdivision $(P,\cS)$. %and assume that $P$ is a reflexive polytope. 
Let $X$ denote the closure of $X^\circ$ in the projective toric variety over $\K$ corresponding to the normal fan of $P$. We assume that $P$ is \define{reflexive} in the sense of Batyrev \cite[Section~4.1]{BatDual}. That is, we
assume that $P$ contains the origin in its relative interior, and the associated dual polytope $P^*$ is also a lattice polytope. In this case, there is an inclusion-reversing correspondence 
between faces $Q$ of $P$ and faces $Q^*$ of $P^*$. Moreover, $X_{\gen}$ is a projective Calabi-Yau variety with at worst canonical singularities. Similarly, let $X^*$ denote a family of 
projective Calabi-Yau varieties corresponding to the pair $(P^*,\cS^*)$, for some polyhedral subdivision $\cS^*$ of $P^*$. 

Batyrev introduced the notion of a \define{stringy invariant} $E_{\st}(V;u,w)$ of a complex variety $V$ with at worst log-terminal singularities in \cite{BatNon}, such that if $V$ admits a crepant resolution $V'$ then 
$E_{\st}(V;u,w) = E(V';u,w)$. In  \cite[Theorem~7.2]{BMString}, Borisov and Mavlyutov  proved a result equivalent to the following formula for the projective complex variety $V = X_{\gen}$:
\[
uvw^2 E_{\st}(X_{\gen};u,w) = \sum_{Q \subseteq P} (-w)^{\dim Q + 1} \lc(Q,\cS|_Q;uw^{-1}) \lc(Q^*;uw). 
\]
This formula greatly simplified an earlier formula of Batyrev and Borisov \cite[Theorem~4.14]{BBMirror}. Moreover, it follows immediately that  $E_{\st}(X_{\gen};u,w) = u^{\dim P - 1} E_{\st}(X_{\gen}^*;u^{-1},w)$,
which is precisely Batyrev and Borisov's mirror symmetry construction for Calabi-Yau hypersurfaces  \cite[Theorem~4.14]{BBMirror}.

One may extend the definition of stringy invariants to varieties over $\K$, 
and define a polynomial $E_{\st}(X;u,v,w)$, which agrees with the refined limit Hodge-Deligne polynomial of a crepant resolution of $X$ over $\K$.  Using the methods of  \cite{BMString}, together with Theorem~\ref{t:mainhyper}, yields the formula
\[
uvw^2 E_{\st}(X;u,v,w) = \sum_{Q \subseteq P} (-w)^{\dim Q + 1} \lc(Q,\cS|_Q;u,v) \lc(Q^*;uvw^2). 
\]
We observe that there is no direct relation between $E_{\st}(X;u,v,w)$ and $E_{\st}(X^*;u,v,w)$, except in the case when both $X$ and $X^*$ are either trivial degenerations or maximally degenerate, in the sense 
that all non-zero limit mixed Hodge numbers are of type $(p,p)$,  in which case one recovers a statement equivalent to the Batyrev-Borisov result above. 

\end{example}

\section{Intersection cohomology of sch\"{o}n subvarieties} \label{s:intcoh}

%\eric{opening paragraph changed}

In this section, we give a sum-over strata formula for the refined limit Hodge-Deligne polynomial of the intersection cohomology of the closure of a sch\"{o}n subvariety in certain projective toric varieties over a punctured curve $C$. This formula is analogous to a special case of the motivic formula obeyed by the usual Hodge-Deligne polynomial.  It differs in that it only works for stratifications induced by the ambient toric variety and that it requires a weighting of terms by the $g$-polynomial to account for singularities along strata.  By considering the case of sch\"{o}n hypersurfaces, we will give an alternative proof of Theorem \ref{t:mainhyper} for families of sch\"on hypersurfaces over a punctured curve.

We will let $\K=\C(t)$ instead of the function field of germs of analytic functions on a punctured disc. We view varieties defined over $\K$ as algebraic families of varieties over a curve $C$ that has a distinguished puncture $0$.  All monodromy will be computed around this puncture.

\subsection{Sum-over-strata formulas in intersection cohomology}

In the proof \cite{BBMirror} of their formula for the cohomology of a sch\"on hypersurface of a toric variety,  Batyrev and Borisov observe that the intersection cohomology of sch\"on hypersurfaces in the projective toric variety associated to their Newton polytope obeys a sum-over-strata formula analogous to that of the cohomology of projective toric varieties. Cappell, Maxim, and Shaneson \cite{IntCoh} who study what they call the `stratified multiplicative property of intersection cohomology'  prove a natural generalization of that observation.
 They study an intersection cohomology Euler characteristic (such as topological Euler characteristic,  $\chi_y$-characteristic, or Hodge-Deligne polynomial), extend its definition to open strata and study how it behaves under a stratified fibration $f:X\rightarrow Y$.   
One can generalize these results to give a sum-over-strata formula for the refined limit Hodge-Deligne polynomial of a family of sch\"{o}n subvarieties over a punctured disc.

We first establish our framework.  
We will use middle-perversity throughout.  All stratification will be complex algebraic stratifications.
The compactly supported intersection cohomology of a quasi-projective variety over $\C$ has a mixed Hodge structure \cite{Saito} and therefore one can define a Hodge-Deligne polynomial.   For a  quasi-projective variety $X^\circ$ over a curve $C$, the intersection cohomology with compact supports of the family forms a mixed Hodge module by the work of Saito \cite{SaitoPolarisable}.  In an arbitrarly small disc  around the puncture, we can suppose that this is an admissible variation of mixed Hodge structures.  Therefore we have a Hodge, weight, and monodromy-weight filtration on the compactly supported intersection cohomology of a generic fiber, and, as in Section~\ref{s:motivicfamily}, we can form the refined limit Hodge-Deligne polynomial and refined limit mixed Hodge numbers.  The advantage of using intersection cohomology is that projective varieties carry a {\em pure} Hodge structure and that toric strata can be treated as if they are smoothly embedded once one uses a combinatorial correction term coming from the $g$-polynomial.  Let $X^\circ$ be a sch\"{o}n subvariety of $(\K^*)^n$.  Let $\Delta_P$ be the normal fan of a lattice polytope $P$ such that the recession fan of $\Trop(X^\circ)$ is supported on $\Delta_P$ i.e. the support of the recession fan is 
a union of cones in $\Delta_P$. Such a polytope always exists by arguments using the Hilbert scheme \cite{TevComp}.  
In this case, we say $\P(\Delta_P)$ is \define{adapted} to $X^\circ$.  If $X^\circ$ is a sch\"{o}n hypersurface, it suffices to take $P$ to be the Newton polytope of $X^\circ$.  We will let $X$ be the closure of $X^\circ$ in $\P(\Delta_P)_\K$. We say that $X$ is a \define{sch\"on}, projective variety. Note that for $X^\circ=(\K^*)^n$, we have $X=\P(\Delta_P)_\K$.  We may also study the case where $X^\circ$ is the sch\"{o}n subvariety of $(\C^*)^n$.  In this case, we say $\P(\Delta_P)_\K$ is adapted if it is adapted to $X^\circ\times_{\C}\K$.

We begin with the analogue of the sum-over-strata formula analogous to the motivic formula for compactly supported cohomology.  Please note that our convention for the $g$-polynomial differs from that of \cite{BBMirror}. 

\begin{theorem} \label{t:sumoverstrata}  Let $X\subseteq \P(\Delta_P)_{\K}$ be the closure of a sch\"{o}n subvariety in an adapted projective toric variety.  
%If it is defined over $\C$, the Hodge-Deligne polynomial obeys
%\[
%E_{\inter}(X_{\gen};u,w)=\sum_{\sigma\in\Delta_P} E(X^\circ_\sigma;u,w)g([0,\sigma];uw).
%\]
%f it is defined over $\K$,
The refined limit Hodge-Deligne polynomial obeys
\[E_{\inter}(X_\infty;u,v,w)=\sum_{\sigma\in\Delta_P} E((X_\sigma^\circ)_{\infty};u,v,w)g([0,\sigma];uvw^2).
\]
\end{theorem}

\begin{proof}
The analogous formula for the Hodge-Deligne polynomial for varieties over $\C$ is deduced from the decomposition theorem of Beilinson, Bernstein, Deligne and Gabber \cite{BBD} in \cite[Corollary 3.17]{BBMirror}. There it is stated for hypersurfaces, but the arguments also hold for the closure of sch\"{o}n subvarieties in adapted projective toric varieties.  Work in a similar direction has been done by Cappell, Maxim, and Shaneson \cite{IntCoh} who study a `stratified multiplicative property' which is proved for the Hodge-Deligne polynomial in intersection cohomology.  Again, the result is an application of the decomposition theorem.   One can derive the sum-over-strata formula from the stratified multiplicative property as follows: one takes a projective toric resolution of singularities $f:\P({\Delta}_{\tilde{P}})\rightarrow\P(\Delta_{P})$; there is an induced resolution of singularities of the closures of the sch\"{o}n subvariety, $f:\tilde{X}\rightarrow X$ to which one applies the stratified multiplicative property; and one then deduces the sum-over-strata formula from the analogous formula on $\tilde{X}$ where the intersection cohomology Hodge-Deligne polynomial reduces to the usual Hodge-Deligne polynomial which is known to be motivic. 

%One can also use the approach of \cite{IntCoh} by  picking a resolution of singularities induced from the ambient toric variety, writing down the motivic sum formula for ordinary cohomology on the resolution of singularities and then using the decomposition theorem.

To justify the formula for the refined limit Hodge-Deligne polynomial, we use the approach of \cite{IntCoh}.  Their results are stated for stratifications with simply connected strata but they note that they only need the property that the local systems involved in the decomposition theorem have trivial monodromy along strata.  This property is verified for sch\"on hypersurfaces in \cite[Cor 3.17]{BBMirror}.  The same proof holds for sch\"{o}n subvarieties.  To obtain the result for the limit mixed Hodge structure coming from a family, it suffices to show that the isomorphism in the decomposition theorem respects the monodromy-weight filtration.

Let $X_t$ be a family of sch\"{o}n subvarieties of a projective toric variety $\P(\Delta_P)_C$ over a (possibly non-proper) curve $C$.  Here, we will be concerned with the monodromy around $0$.
 Write $p:X\rightarrow C$.  Take a toric resolution of singularities $f:\P(\Delta_{\tilde{P}})_C\rightarrow \P(\Delta_P)_C$, and let $\tilde{X}$ be the closure of $X^\circ$ in $\P(\Delta_{\tilde{P}})$.   By applying the decomposition theorem to the resolution of singularities
 $f:\tilde{X}\rightarrow X$, we have a non-canonical isomorphism,
\begin{eqnarray} \label{mdecomp}
(p\circ f)_*\Q_{\tilde{X}}[n+1]&\cong& \bigoplus_i \bigoplus_{l=0}^{\dim Y} p_*(IC_{\overline{S_l}}(L_{i,l})[-i])
\end{eqnarray}
where we have stratified the map $f$ by $X=\coprod_l S_l$, $0\leq l\leq \dim X$, $\alpha_l:S_l\hookrightarrow X$ and the local system are given by $L_{i,l}=\alpha_l^*\cH^{-l}(\cHp^i(f_*\Q_X[n]))$.   In this case, the stratification coincides with that induced by the ambient toric variety and the local systems $L_{i,l}$ are equal to those that occur in the decomposition theorem applied to $f:\P(\Delta_{\tilde{P}})_C\rightarrow\P(\Delta_P)$ and are therefore constant.
%Now by applying the existence theorem for Whitney stratifications (\cite[p.43] {GMSMT}) together with the Thom isotopy lemma over a sufficiently small punctured disc around the origin, we can suppose that $p|_{S_i}$ is a locally trivial fibration.  
Consequently, the cohomology sheaves of all terms in (\ref{mdecomp}) give local families in a punctured disc around around $0$.  Therefore, we may write down a monodromy operator and form the weight-monodromy filtration which is compatible with the isomorphism.  
The sheaf $\Q_{\tilde{X}}$ has the structure of a Hodge module, hence by Saito's theory, the derived pushforwards $f_*\Q_{\tilde{X}}$ and $(p\circ f)_*\Q_{\tilde{X}}$ have the structure of mixed Hodge modules.  Likewise, the relevant cohomology sheaves carry the structure of a mixed Hodge module \cite{SaitoPolarisable}, so their pushforwards do as well.  See \cite[Section 8.3.3]{RTT} for an exposition. Consequently, the formula (\ref{mdecomp}) is an isomorphism of mixed Hodge modules over $C$. 
Therefore, over a small punctured
disc about $0$, we get an isomorphism
of admissible variations of mixed
Hodge structures.
\end{proof}

Note that the sum in the above theorem only needs to be over cones of $\Delta_P$ in the support of the recession fan of $\Trop(X^\circ)$ because for other cones $\sigma$, $X^\circ_\sigma$ is empty and does not contribute.

\subsection{Intersection cohomology of families of sch\"on hypersurfaces of toric varieties}

We will compute the refined limit Hodge-Deligne polynomial of sch\"on, projective hypersurfaces using intersection cohomology.  Our proof is inspired by that of \cite{BBMirror} where one sums over strata in the stratification induced by the ambient toric variety and then constrains the intersection cohomology by Poincar\'{e} duality and the weak Lefschetz theorem.

Let $V\subset \P(\Delta_P)_{\C}$ be a closure of a sch\"{o}n hypersurface in an adapted projective toric variety  defined by a lattice polytope $P$.
We will need the following observations about the intersection cohomology of $V$:
\begin{enumerate}
\item[(a)] The intersection cohomology of $V$ obeys Poincar\'{e} duality \cite{IH2}. 

\item[(b)] The Hodge structure on $\IH^*(V)$ is pure \cite{SaitoPolarisable}.

\item[(c)] By the weak Lefschetz  theorem, the Gysin map $\IH^{k}(V)\rightarrow \IH^{k+2}(\P(\Delta)_\C)$ is a surjective map if $k\geq\dim V$ and is an isomorphism if $k>\dim V$.  Moreover, it is a morphism of Hodge structures \cite{IH2}.
\end{enumerate}

Now, we consider a family of closures of sch\"{o}n hypersurfaces $X_t$  in $\P=\P(\Delta)_{\C}$ over  a curve $C$ 
that has a distinguished puncture
such that $\P(\Delta_P)_\C$ is adapted for each $X_t$. By naturality, the monodromy around the puncture commutes with Poincar\'{e} duality and the Gysin map.  We write $X = X_{\K}$ for $X_t$ considered as a subvariety of $\P(\Delta_P)$ over $\K$.

We begin by writing down the refined limit Hodge-Deligne polynomials for $\P(\Delta_P)_{\K}$ and $X$.  Because for each face $Q$ of $P$, we have $\P(\Delta_P)^\circ_Q=(\K^*)^{\dim Q}$ and
$E((\P(\Delta_P)^\circ_Q)_\infty;u,v,w)=(uvw^2-1)^{\dim Q}$, we have from Theorem \ref{t:sumoverstrata}:
\begin{eqnarray*}
E_{\inter}(\P_\infty;u,v,w)&=&\sum_{\substack{Q\subseteq P\\Q\neq \emptyset}} (uvw^2-1)^{\dim Q}g([Q,P]^*;uvw^2)\\
E_{\inter}(X_\infty;u,v,w)&=&\sum_{\substack{Q\subseteq P\\Q\neq \emptyset}} E((X^\circ_Q)_\infty;u,v,w)g([Q,P]^*;uvw^2).
\end{eqnarray*}

Note that the first formula  shows that $E_{\inter}(\P_\infty;u,v,w) = f(uvw^2)$, where $f$ is the \emph{toric $h$-polynomial} of $P$ which is well-known to give the dimensions of the topological intersection cohomology of the toric variety $\P$ (see, for example, \cite{BraRemarks,F}).

We define $E_{\inter,\Lef}$ by
\begin{eqnarray*}
E_{\inter,\Lef}(X_\infty;u,v,w) &=&   \sum_{m=0}^{\dim P-1}\sum_{p,q,r} (-1)^m h^{p,q,r}(IH^m(\P_{\infty})) u^p v^q w^r \\
&+& \sum_{m=\dim P}^{2\dim P-2}\sum_{p,q,r} (-1)^m h^{p+1,q+1, r+ 2}(IH^{m+2}(\P_\infty)) u^p v^q w^r.
\end{eqnarray*}
Note that this is a polynomial in $uvw^2$.
 It represents the cohomology of $X$ that we know must exist by the weak hyperplane theorem and  Poincar\'{e} duality. 
Because all the relevant cohomology is of type $(p,p)$, the action of monodromy must be trivial by 
Remark~\ref{r:zero}.  Then from the above expression for $E_{\inter}(\P_\infty;u,v,w)$ and Definition~\ref{d:g}, one may deduce that
\[
E_{\inter,\Lef}(X_{\infty};u,v,w) = E_{\inter,\Lef}(P;uvw^2), 
\]
 where $E_{\inter,\Lef}(P;t)$ is defined by 
\begin{equation}\label{e:Lef}
(t - 1)E_{\inter,\Lef}(P;t) = t^{\dim P}g([\emptyset,P]^*;t^{-1})- g([\emptyset,P]^*;t). 
\end{equation}
We define 
\[E_{\inter,\prim}(X_\infty;u,v,w)=E_{\inter}(X_\infty;u,v,w)-E_{\inter,\Lef}(X_{\infty};u,v,w).\]
This is the refined limit Hodge-Deligne polynomial corresponding to the primitive intersection cohomology in degree $\dim P-1$, 
\[IH^{\dim P - 1}_{\prim}(X_\infty):= \ker[IH^{\dim P - 1}(X_\infty)  \rightarrow IH^{\dim P + 1} (\P_\infty)].\]
As in the proof of \cite[Proposition~3.22]{BBMirror}, the induced Hodge structure $(F,W)$ is pure and concentrated in $W$-degree $\dim P-1$.

%%%%%%%%%%%%%%%%%%%%%

The following lemma establishes our main result (Corollary~\ref{c:mainintersection}) on the intersection cohomology of families of sch\"on, projective varieties. Explicitly, 
let $X^\circ \subseteq (\K^*)^n$ be a sch\"{o}n hypersurface, with associated Newton polytope and polyhedral subdivision $(P,\cS)$ and $\dim P = n$. Let $X$ denote the closure of $X^\circ$ in $\P(\Delta_P)_\K$.
Then Theorem~\ref{t:mainhyper}, which we proved in Section~\ref{ss:refinedhyper}, states that
\begin{equation}\label{e:main1}
uvw^2E(X_\infty^\circ;u,v,w) = (uvw^2 - 1)^{\dim P} + (-1)^{\dim P + 1}h^*(P,\cS;u,v,w),
\end{equation}
while Corollary~\ref{c:mainintersection} states that
\begin{equation}\label{e:main2}
uvw^2E_{\inter}(X_\infty;u,v,w) = uvw^2E_{\inter,\Lef}(P;uvw^2) + (-1)^{\dim P + 1}\lc(P,\cS;u,v)w^{\dim P + 1},
\end{equation}
where $E_{\inter,\Lef}(P;uvw^2)$ is given by \eqref{e:Lef}.

\begin{lemma}\label{l:equiv}
Theorem~\ref{t:mainhyper} and Corollary~\ref{c:mainintersection} are equivalent. 
\end{lemma}
\begin{proof}
%Theorem~\ref{t:mainhyper} states that
%\begin{equation}\label{e:main1}
%uvw^2E(X_\infty^\circ;u,v,w) = (uvw^2 - 1)^{\dim P} + (-1)^{\dim P + 1}h^*(P,\cS;u,v,w).
%\end{equation}
%while Corollary~\ref{c:mainintersection} states that
%\begin{equation}\label{e:main2}
%uvw^2E_{\inter}(X_\infty;u,v,w) = uvw^2E_{\inter,\Lef}(P;uvw^2) + (-1)^{\dim P + 1}\lc(P,\cS;u,v)w^{\dim P + 1},
%\end{equation}
%where $E_{\inter,\Lef}(P;uvw^2)$ is given by \eqref{e:Lef}. 
Firstly, we show that \eqref{e:main1} implies \eqref{e:main2}. Indeed, for every non-empty face $Q$ of $P$, \eqref{e:main1} implies  
\[
uvw^2E((X_Q^\circ)_\infty;u,v,w) = (uvw^2 - 1)^{\dim Q} + (-1)^{\dim Q + 1}h^*(Q,\cS|_Q;u,v,w).
\]
We multiply this equation by $g([Q,P]^*;uvw^2)$ and sum over all such $Q$ to obtain
\[
E_{\inter}(X_\infty;u,v,w) = \sum_{\substack{Q\subseteq P\\Q\neq \emptyset}} \big[(uvw^2 - 1)^{\dim Q} + (-1)^{\dim Q + 1}h^*(Q,\cS|_Q;u,v,w) \big]g([Q,P]^*;uvw^2),
\]
where the left hand side is computed using Theorem~\ref{t:sumoverstrata}.  
Adding the equation $0 = -g([\emptyset,P]^*;uvw^2) + g([\emptyset,P]^*;uvw^2)$, and then simplifying using Definition~\ref{d:g} and Theorem~\ref{t:Eulerianinverse}, yields \eqref{e:main2} as desired. 
We obtain \eqref{e:main1} from \eqref{e:main2} similarly. Explicitly, we claim that 
\[
E(X^\circ_\infty;u,v,w) =\sum_{\substack{Q\subseteq P\\Q\neq \emptyset}} E_{\inter}((X_Q)_\infty;u,v,w) (-1)^{\dim P - \dim Q} g([Q,P];uvw^2).
\]
This follows by applying Theorem~\ref{t:sumoverstrata} to the right hand side, together with Theorem~\ref{t:Eulerianinverse}. 
Now the above argument holds with $g([Q,P]^*;uvw^2)$ replaced by $(-1)^{\dim P - \dim Q} g([Q,P];uvw^2)$.
\end{proof}

We now give a new proof of \eqref{e:main2} in the following equivalent form:
 \begin{corollary}\label{c:intexplicit}
  Let $\K = \C(t)$ and 
let $X^\circ \subseteq (\K^*)^n$ be a sch\"{o}n hypersurface, with associated Newton polytope and polyhedral subdivision $(P,\cS)$ and $\dim P = n$. Let $X$ denote the closure of $X^\circ$ in $\P(\Delta_P)_\K$.   Then the refined limit Hodge-Deligne polynomial associated to the intersection cohomology of $X$ is given by
%Let $X^\circ \subseteq (\K^*)^n$ be a sch\"{o}n hypersurface, with associated Newton polytope and polyhedral subdivision $(P,\cS)$ and $\dim P = n$. Then the refined limit mixed Hodge numbers
%associated to the primitive cohomology of $X^\circ$ are given by
\begin{equation}\label{e:main3}
uvw^2E_{\inter,\prim}(X_\infty;u,v,w)=(-1)^{\dim P+1}w^{\dim P+1}\lc(P, \cS;u,v) 
\end{equation}
Equivalently,
%The above corollary yields the following formula analogous to Corollary~\ref{c:explicit}:
\[
\lc(P, \cS;u,v) = uv\sum_{p,q} h^{p,q}(IH_{\prim}^{\dim P - 1}(X_\infty)) u^p v^q.  
\]
 \end{corollary}

\begin{proof}
We have a commutative diagram
\[\xymatrix{
\Z[u,v,w]  \ar[r]^{\substack{u \mapsto uw^{-1} \\ v \mapsto 1}}   \ar[d]^{w \mapsto 1}  & \Z[u,w] \ar[d]^{w \mapsto 1}  &\\
\Z[u,v] \ar[r]^{v \mapsto 1} & \Z[u]  \ar[r]^{u \mapsto 1}  & \Z.
}\] 
We will prove $\eqref{e:main3}$ by working our way through the diagram. In fact, we have  proved the specialization of the equivalent statement $\eqref{e:main1}$ to $\Z[u]$ in Theorem~\ref{t:xy}.

In Section~\ref{ss:refinedhyper}, we proved that if $\eqref{e:main1}$ holds when specialized to $\Z[u,w]$ then $\eqref{e:main1}$ holds when specialized to 
$\Z[u,v]$.  Hence, we are left with the vertical arrows of the diagram. 
Because primitive cohomology is concentrated in $W$-degree equal to $\dim P-1$, it is clear that if $\eqref{e:main3}$ holds for $\Z[u]$, then it holds for $\Z[u,w]$.   Similarly, if $\eqref{e:main3}$ holds for $\Z[u,v]$, then it holds for $\Z[u,v,w]$. 
\end{proof}

\bibliographystyle{amsplain}
\def\cprime{$'$}
\providecommand{\bysame}{\leavevmode\hbox to3em{\hrulefill}\thinspace}
\providecommand{\MR}{\relax\ifhmode\unskip\space\fi MR }
\providecommand{\MRhref}[2]{%
  \href{http://www.ams.org/mathscinet-getitem?mr=#1}{#2}
}
\providecommand{\href}[2]{#2}

\end{document}